\documentclass[12pt,a4paper]{amsart}

\newif{\ifDierk}\Dierkfalse

\ifDierk\else
\usepackage[pdftex]{epsfig}

\usepackage{charter,euler} 
\DeclareSymbolFont{bchoperators}{T1}{bch}{m}{n}
\SetSymbolFont{bchoperators}{bold}{T1}{bch}{b}{n}
\makeatletter
\renewcommand{\operator@font}{\mathgroup\symbchoperators}
\makeatother
\fi

\usepackage{fullpage}
\usepackage{graphicx}

\usepackage{amssymb}  
\usepackage{latexsym} 
\usepackage{comment}
\usepackage{color}
\usepackage{enumerate}
\usepackage[all]{xy}

\usepackage{titlesec} 
\titleformat{\section}{\normalfont\bfseries\filcenter}{\thesection}{1em}{}
\titleformat{\subsection}{\normalfont\bfseries}{\thesubsection}{1em}{}
\titleformat{\subsubsection}{\normalfont\bfseries}{\thesubsubsection}{1em}{}

\newcommand{\C}{{\mathbb C}}
\newcommand{\BP}{{\mathbb P}}
\newcommand{\Q}{{\mathbb Q}}
\newcommand{\R}{{\mathbb R}}
\newcommand{\Sphere}{{\mathbb S}}
\newcommand{\Z}{{\mathbb Z}}
\newcommand{\CA}{{\mathcal A}}
\newcommand{\CP}{{\mathcal P}}
\newcommand{\To}{\longrightarrow}
\newcommand{\Aff}{\operatorname{Aff}}
\newcommand{\eps}{\varepsilon}
\renewcommand{\Re}{\operatorname{Re}}
\renewcommand{\Im}{\operatorname{Im}}
\newcommand{\uz}{\underline{z}}
\newcommand{\sm}{\setminus}
\newcommand{\TotalDiff}[2]{\mathbf{D}(#1)_{|#2}}

\newtheorem{Theorem}{Theorem}[section]
\newtheorem{Maintheorem}{Theorem}
\newtheorem{Lemma}[Theorem]{Lemma}
\newtheorem{Proposition}[Theorem]{Proposition}

\theoremstyle{definition}

\newtheorem{Remark}[Theorem]{Remark}

\numberwithin{equation}{section}

\definecolor{darkgreen}{rgb}{0,0.5,0}

\usepackage[
        colorlinks, citecolor=darkgreen,
        backref,
        pdfauthor={Bernhard Reinke, Dierk Schleicher, Michael Stoll},
        pdftitle={The Weierstrass root finder is not generally convergent},
]{hyperref}

\usepackage[alphabetic,backrefs,lite]{amsrefs} 

\begin{document}

\title{The Weierstrass root finder is not generally convergent}
\author{Bernhard Reinke}
\address{Aix-Marseille Universit\'e,
         Institut de Math\'ematiques de Marseille,
         163 Avenue de Luminy Case 901,
         13009 Marseille, France.}
\email{bernhard.reinke@univ-amu.fr}

\author{Dierk Schleicher}
\address{Aix-Marseille Universit\'e,
         Institut de Math\'ematiques de Marseille,
         163 Avenue de Luminy Case 901,
         13009 Marseille, France.}
\email{dierk.schleicher@gmx.de}

\author{Michael Stoll}
\address{Mathematisches Institut,
         Universit\"at Bayreuth,
         95440 Bayreuth, Germany.}
\email{Michael.Stoll@uni-bayreuth.de}

\date{April 9, 2020}


\begin{abstract}
  Finding roots of univariate polynomials is one of the fundamental tasks
  of numerics, and there is still a wide gap between root finders that are
  well understood in theory and those that perform well in practice.
  We investigate the root finding method of Weierstrass, a root finder
  that tries to approximate all roots of a given polynomial in parallel
  (in the Jacobi version, i.e., with parallel updates). This method has a
  good reputation for finding all roots in practice except in obvious cases
  of symmetry, but very little is known about its global dynamics and
  convergence properties.

  We show that the Weierstrass method, like the well known Newton method,
  is not generally convergent: there are
  open sets of polynomials~$p$ of every degree $d \ge 3$ such that
  the dynamics of the Weierstrass method applied to~$p$ exhibits attracting
  periodic orbits.  Specifically, all polynomials
  sufficiently close to $Z^3 + Z + 180$ have attracting cycles of period~$4$.
  Here, period $4$ is minimal: we show that for cubic polynomials, there are no
  periodic orbits of length $2$ or~$3$ that attract open sets of starting
  points.

  We also establish another convergence problem for the Weierstrass method:
  for almost every polynomial of degree $d\ge 3$ there are orbits
  that are defined for all iterates but converge
  to~$\infty$; this is a problem that does not occur for Newton's method.

  Our results are obtained by first interpreting the original problem
  coming from numerical mathematics in terms of higher-dimensional
  complex dynamics, then phrasing the question in algebraic terms in such a
  way that we could finally answer it by applying methods from computer
  algebra.
\end{abstract}

\maketitle



\section{Introduction} 

Finding roots of polynomials is one of the  fundamental tasks in mathematics,
highly relevant in theory for many fields as well as numerous practical
applications. Since the work of Ruffini--Abel, it is clear that in general
the roots cannot be found by finite radical extensions, so numerical
approximation methods are required. One may find it surprising that, despite
age and relevance of this problem, no clear algorithm is known that has a
well-developed theory and works well in practice.

There are ``algorithms'' (in the sense of heuristics) that seem to work
in practice fast and reliably, among them the Weierstrass and Ehrlich--Aberth
methods, which are both iterations in as many variables as the number of
roots to be found, and which are supposed to converge to a vector of roots
under iteration. They are known to converge quadratically resp.\ cubically
near the roots, but have essentially no known global theory. Then there are
algorithms such as Pan's that have excellent theoretical complexity (optimal
up to log-factors), but they cannot be used in practice because of their
lack of stability.

An interesting method is Newton's, which may well be the best-known method;
it approximates one root at a time. This is a simple method that is stable
and converges quadratically near the roots, so it is often used to polish
approximate roots. However, it is an iterated rational map, so it is
``chaotic'' on its Julia
set, and its global dynamics is hard to describe. In particular, it is
well known to be not generally convergent: there are open sets of polynomials
and open sets of starting points on which the Newton dynamics does not
converge to any root, but rather to an attracting periodic orbit
(``an attracting cycle'') of period $2$ or higher. Its use has thus often
been discouraged. However, in recent years quite some theory has been
developed about its global dynamics and its expected (rather efficient) speed
of convergence. At the same time, it has been used in practice successfully
to find all roots of polynomials of degree exceeding~$10^9$ in remarkable
speed. Some of these results are described in Section~\ref{Sec:Newton}.
Therefore, Newton's method stands out as one that at the same time has good
theory and performs well in practice.

The focus of our work is on the Weierstrass iteration method, also known as
the Durand--Kerner-method. For this method, we are not aware of any global
theory of its dynamics, but it is well known in practice to find all roots
of a complex polynomial in all cases, except in the presence of obvious
symmetries: for instance, when the polynomial is real but some of its roots
are not, then any purely real vector of starting points cannot converge to
the roots, since the method respects complex conjugation.

Our first result says that this observation does not hold in general.

\begin{Maintheorem}[The Weierstrass method is not generally convergent] \strut
  \label{MT1}
  \begin{enumerate}[\upshape(1)]
   \item \label{Item:MT1-NotGenConv}
          There is an open set of polynomials~$p$ of every degree
          $d \ge 3$ such that the (partially defined) Weierstrass iteration
          $W_p \colon \C^d \to \C^d$ associated to~$p$  has attracting cycles of period $4$.
          In particular, Weierstrass's method is not generally
          convergent for polynomials of degree at least~$3$.
    \item \label{Item:MT1-CubicOkay}
          Period $4$ is minimal with this property: for every cubic
          polynomial~$p$ the associated Weierstrass iteration
          $W_p \colon \C^3 \to \C^3$ associated to~$p$ cannot have
          an attracting cycle of period $2$ or~$3$.
  \end{enumerate}
\end{Maintheorem}

This theorem answers in the affirmative a question asked by Steve Smale:
he expected the existence of attracting cycles in the 1990's, if not earlier,
in analogy to the Newton dynamics (Victor Pan, personal communication).

Following  McMullen~\cite{McMullenRootFinding},
we say that a root-finding method is \emph{generally convergent} if, for an
open dense set of polynomials of fixed degree, there is an open dense
set of starting points in~$\C$ that converge to one of the roots. To our
knowledge, the only way to establish failure of general convergence is to
find a polynomial~$p$ that, under the given iteration method, has an
attracting periodic orbit (an ``attracting cycle'') of period $n \ge 2$.
This attracting cycle must attract a neighborhood of the cycle, and it would
persist under small perturbations of~$p$, so convergence to a root fails on
an open set of starting points for an open set of polynomials. Therefore,
our theorem establishes that the Weierstrass method is not generally
convergent for polynomials of degrees $3$ or higher. (Other ways of failure
of general convergence are of course conceivable but have apparently never
been observed).

It is well known that the Weierstrass method has another problem: some
orbits are not defined forever. The Weierstrass method $W_p \colon \C^d \to \C^d$
is not defined whenever two coordinates in~$\C^d$ coincide; this problem
may occur even after any number of iteration steps from a starting vector with
distinct entries.

Our second main result establishes the existence of
a very different kind of problem for the
Weierstrass method that apparently was not known: there are  orbits in~$\C^d$
for which the iteration is always defined that converge to~$\infty$ (in the
sense that the orbit leaves every compact subset of~$\C^d$). This problem
exists (at least) for every polynomial of degree $d \ge 3$ that has
only simple roots. In fact, we prove a slightly stronger result;
see Section~\ref{Sec:escape}.

\begin{Maintheorem}[The Weierstrass method has escaping points]
  \label{MT2}
  For every polynomial~$p$ of degree $d \ge 3$ with only simple roots,
  there are vectors in~$\C^d$ whose orbits under~$W_p$ tend to infinity.
  The set of escaping points contains a holomorphic curve.
\end{Maintheorem}

It might be interesting to observe that this problem does not exist for
Newton's method: here, $\infty$ is a ``repelling fixed point'', and all
points sufficiently close to~$\infty$ will always iterate closer toward
the roots. For degenerate  polynomials like $Z \mapsto Z^d$, all Newton
orbits converge to the single root, while Weierstrass has escaping orbits
even for $Z \mapsto Z^3$ (see Remark~\ref{Rem:CubicEscaping}).

We cannot resist stating an analogy to the dynamics of transcendental
entire functions in one complex variable: all such functions have escaping
points (points that converge to~$\infty$ under iteration);
see~\cite{EremenkoEscaping}.
Already Fatou observed that in many cases, the set of escaping points
contains curves to~$\infty$; in the 1980's Eremenko raised the
conjecture that all escaping points were on such curves to~$\infty$.
This conjecture was established for many classes of entire functions,
and disproved in general, in~\cite{RRRS}.
It  is plausible that the set of escaping points for~$W_p$ has the
following property: every escaping point can be joined with~$\infty$
by a curve consisting of escaping points.

There is a substantial body of literature on root finding in general,
and on background on our methods in particular. We just mention the survey
article by Pan~\cite{PanSolvingPolynomials} about various known methods and
their properties, as well as the references therein and the surveys by
McNamee~\cites{McNamee1,McNamee2}.


\subsection*{Structure of this paper}

In Section~\ref{Sec:Newton}, we describe some background on Newton's method
and its properties, in order to describe analogies and to build up some
intuition. Basic properties of the Weierstrass method are then described
in Section~\ref{Sec:Weierstrass}. In particular, we discuss escaping points
for the Weierstrass method, starting with the simple polynomial $Z \mapsto Z^3$,
and give a proof of Theorem~\ref{MT2}.

In Section~\ref{Sec:PeriodicPoints}, we describe some algebraic properties
of the Weierstrass method and its periodic points. In the final
Section~\ref{Sec:Cubic} we focus on the case of cubic polynomials, giving
an explicit description of periodic points of low periods; in particular we
give  a proof of Theorem~\ref{MT1}.


\subsection*{Notation and conventions}

All our polynomials will be univariate and over the complex numbers, so we
have polynomials $p \in \C[Z]$ (the indeterminate variable will usually be
called~$Z$). The associated Newton map is denoted~$N_p$, the Weierstrass
map~$W_p$. In general, we denote the $n$-th iterate of a map $F$
by~$F^{\circ n}$. When we want to highlight that a point $z \in \C^d$ is a
vector, we write $\uz$ for $(z_1, \dots, z_d)$. The Jacobi matrix of a map~$F$
at a point~$\uz$ is denoted~$\TotalDiff{F}{\uz}$.

A polynomial $p \in \C[Z]$ is \emph{monic} if its leading coefficient
equals~$1$; that means, if the roots of~$p$ are $\alpha_1, \dots, \alpha_d$,
that $p(Z) = \prod_j (Z-\alpha_j)$. It turns out that both for Newton and for
Weierstrass, it is sufficient to consider monic polynomials.


\subsection*{Acknowledgments}

We gratefully acknowledge support by the European Research Council in the
form of the ERC Advanced Grant~HOLOGRAM.

This research was inspired by discussions with Dario Bini and Victor Pan,
and it completes work initiated jointly with Steffen Maass (now at University
of Michigan). We are grateful for the discussions we had with them, as well
as with other members of our research team.

The algebraic computations described in this paper were done using Magma~\cite{Magma}
and Singular~\cite{Singular}; further numerical computations were done using HomotopyContinuation.jl~\cite{HomotopyContinuation}.


\section{Newton's method and its properties}
\label{Sec:Newton}

Even though the main results in this paper are about the Weierstrass method,
we provide a review of the Newton method in order to build up intuition and
explain analogies, especially since some of these analogies were guiding us
in our research. Interestingly, much more is known about the global dynamics
of Newton's method than about the Weierstrass method.

Newton's method is perhaps the most classical root finding method.
One of its virtues is its simplicity: to find roots of a monic polynomial
$p(Z) = \prod_j (Z-\alpha_j)$, update any approximation $z \in \C$
to a root by
\begin{equation} \label{Eq:NewtonIteration}
  N_p(z) = z - \frac{p(z)}{p'(z)} = z - \left(\sum_j \frac{1}{z-\alpha_j}\right)^{-1}
\end{equation}
and hope that the new number is a better approximation to some root,
at least after a few more iterations. Of course, as long as the roots
of~$p$ are not known, it is the expression in the middle of~\eqref{Eq:NewtonIteration}
that is used to evaluate the Newton iteration.
The right hand side involving the roots cannot be computed, but it may be
helpful in analyzing the properties of the Newton map.
Since only the expression~$p/p'$ enters into the Newton formula,
there is no loss of generality in considering only monic polynomials.

An important property of Newton's method is its compatibility
with affine transformations.
We denote the space of all monic polynomials of degree~$d$ with
complex coefficients by~$\CP'_d$; this is an affine space of dimension~$d$.
It can be identified with~$\C^d$ by taking the coefficients
of~$Z^k$ for $k = 0, 1, \ldots, d-1$ as coordinates. Alternatively,
it can be seen as the quotient $S_d\backslash \C^d$, where $\C^d$
parameterizes the $d$~roots and the symmetric group~$S_d$ acts by permutation
of the coordinates on~$\C^d$.
The group $\Aff(\C)$ of affine transformations of~$\C$ acts on~$\CP'_d$
via its action on the roots of the polynomials.

\begin{Lemma}[Newton's method and affine transformations]
  \label{L:Newton-affine}
  If $p$ is a polynomial and $T \colon \C \to \C$,
  $z \mapsto \alpha z + \beta$, is an affine transformation, then
  \[ N_{Tp} = T \circ N_p \circ T^{-1} \;; \]
  i.e., the Newton dynamics for $p$ and~$Tp$ are affinely conjugate via~$T$.
\end{Lemma}

\begin{proof}
  The defining equation~\eqref{Eq:NewtonIteration} can be written as
  \[ \frac{1}{z - N_p(z)} = \sum_j \frac{1}{z-\alpha_j} \,, \]
  where the~$\alpha_j$ are the roots of~$p$. From this,
  the claim is obvious.
\end{proof}

The lemma above shows that the dynamics of~$N_p$ is conjugate
(and therefore essentially unchanged) if we
replace~$p$ by another polynomial in its orbit under~$\Aff(\C)$.
So the \emph{true parameter space}~$\CP_d$, i.e., the space of polynomial
Newton maps up to affine conjugation, is the quotient of~$\CP'_d$
by the action of~$\Aff(\C)$. This quotient is not a nice space:
the polynomials with a $d$-fold root have a one-dimensional stabilizer
under~$\Aff(\C)$, whereas for all other polynomials, the stabilizer
is finite. This implies that the closure of any point in~$\CP_d$
contains the point~$\bullet$ representing the polynomials with $d$-fold roots.
Removing this point, however, results in a reasonable space, which
has complex dimension $\dim \CP'_d - \dim \Aff(\C) = d - 2$.

There are two fairly natural ways to construct this space.
We can use the action of~$\Aff(\C)$ to move two of the roots
to $0$ and~$1$. The remaining roots form a $(d-2)$-tuple of
complex numbers specifying the polynomial. This representation
is not unique, since we can re-order the roots (and then
normalize the first two roots again). This gives an action of
the symmetric group~$S_d$, and we obtain $\CP_d \sm \{\bullet\} = S_d \backslash \C^{d-2}$.
We can also use the translations in~$\Aff(\C)$ to make the
polynomial \emph{centered}. i.e., such that the sum of the roots
is zero; equivalently, the coefficient of~$Z^{d-1}$ vanishes.
The set of such polynomials can be identified with~$\C^{d-1}$.
This leaves the action of~$\C^\times$ by scaling the roots,
which has the effect of scaling the coefficient of~$Z^k$
by~$\lambda^{d-k}$ (for $k = 0,\ldots,d-2$). Leaving out the
origin of~$\C^{d-1}$ (it corresponds to the ``bad'' polynomials),
we obtain $\CP_d \sm \{\bullet\}$ as the quotient of
$\C^{d-1} \sm \{0\}$ by this $\C^\times$-action. The resulting
space is a weighted projective space of dimension~$d-2$
with weights~$(2,3,\ldots,d)$.

We now fix a period length~$n$. Then the space
\[ \CP_d(n) = \{(p, q) \in \CP_d \times \C : \text{$q$ has period~$n$ under~$N_p$}\} \]
is a finite-degree cover of~$\CP_d$; it particular, it also has dimension~$d-2$.
On~$\CP_d(n)$ we have the holomorphic map
$\mu_{d,n} \colon (p, q) \mapsto (N_p^{\circ n})'(q)$
associating to each point~$q$ of period~$n$ its multiplier.
It is a standard fact that the cycle consisting of~$q$ and its
iterates is \emph{attracting} (i.e., there is an open neighborhood~$U$
of~$q$ such that for all $z \in U$, the sequence $(N_p^{\circ mn}(z))_{m \ge 0}$
converges to~$q$) if and only if $|\mu_{d,n}(p,q)|<1 $.

A great virtue of Newton's method is its fast local convergence: close to a
simple root, the convergence is quadratic, so the number of valid digits
doubles in every iteration step. Therefore, Newton is often employed for
``polishing'' approximate roots (once the roots have been separated from
each other). Yet another virtue is that it can be applied in a great variety
of contexts, in many dimensions as well as for maps that are smooth but
not analytic.

However, Newton's method is not an algorithm but a heuristic: it is a formula
that suggests a hopefully better approximation to any given initial point~$z$.
This formula says little about the properties of the global dynamics,
which is an iterated rational map. As such, it has a Julia set with
``chaotic'' dynamics, and which may well have positive (planar Lebesgue) measure.
Worse yet, Newton's method can have open sets of starting points
that fail to converge to any root, but instead converge to periodic points
of period~$2$ or higher. Therefore Newton's method fails to be generally convergent.
The problems occur even in the simplest possible case: for the cubic polynomial
$p(Z) = Z^3 - 2Z + 2$, the Newton method has an attracting $2$-cycle,
as illustrated in Figure~\ref{Fig:AttractingCycle}. Steven Smale had observed
this phenomenon, and he asked for a classification of such
polynomials~\cite{SmaleQuestion}*{Problem~6 on p.~98}.
Partially in response to this question,
a complete classification of all (postcritically finite) Newton maps of
arbitrary degrees was developed in~\cite{NewtonClassification};
in particular, it implies the following result.

\begin{Proposition}[Polynomials with attracting periodic orbits]
  \label{P:Newton-cycles}
  For every degree $d \ge 2$, the Newton map of a degree~$d$ polynomial
  can have up to $d-2$ attracting periodic orbits that are not fixed points,
  and the periods can independently be arbitrary numbers $2$ or greater.
  This bound is sharp.
\end{Proposition}

This is a rather weak corollary of the general classification result of
postcritically finite Newton maps, in which the dynamics can be prescribed
with far greater precision. Here we give a heuristic explanation.

The upper bound comes from a well-known fact in holomorphic dynamics.
The Newton map~$N_p$ of a polynomial~$p$ with $d$~distinct roots (of possibly higher
multiplicity) is a rational map of degree~$d$, and as such it has $2d-2$ critical
points. Each of the roots of~$p$ is an attracting fixed point and must attract (at least)
one of these critical points, so up to $d-2$ ``free'' critical points remain.
Each attracting cycle of period at least~$2$ must attract one of these critical
points; thus the bound.

For the lower bound, to establish that up to $d-2$ cycles of period at
least~$2$ can be made attracting, the fundamental observation is that the
multipliers of these cycles form a map from $(d-2)$-dimensional parameter space
to a $(d-2)$-dimensional space of multipliers, so under conditions of genericity
one expects this map to have dense image. This will be not so for Weierstrass;
see Section~\ref{Sec:Weierstrass}.

\begin{figure}[htbp]
  \includegraphics[height=.4\textwidth,trim=150 0 0 0,clip]{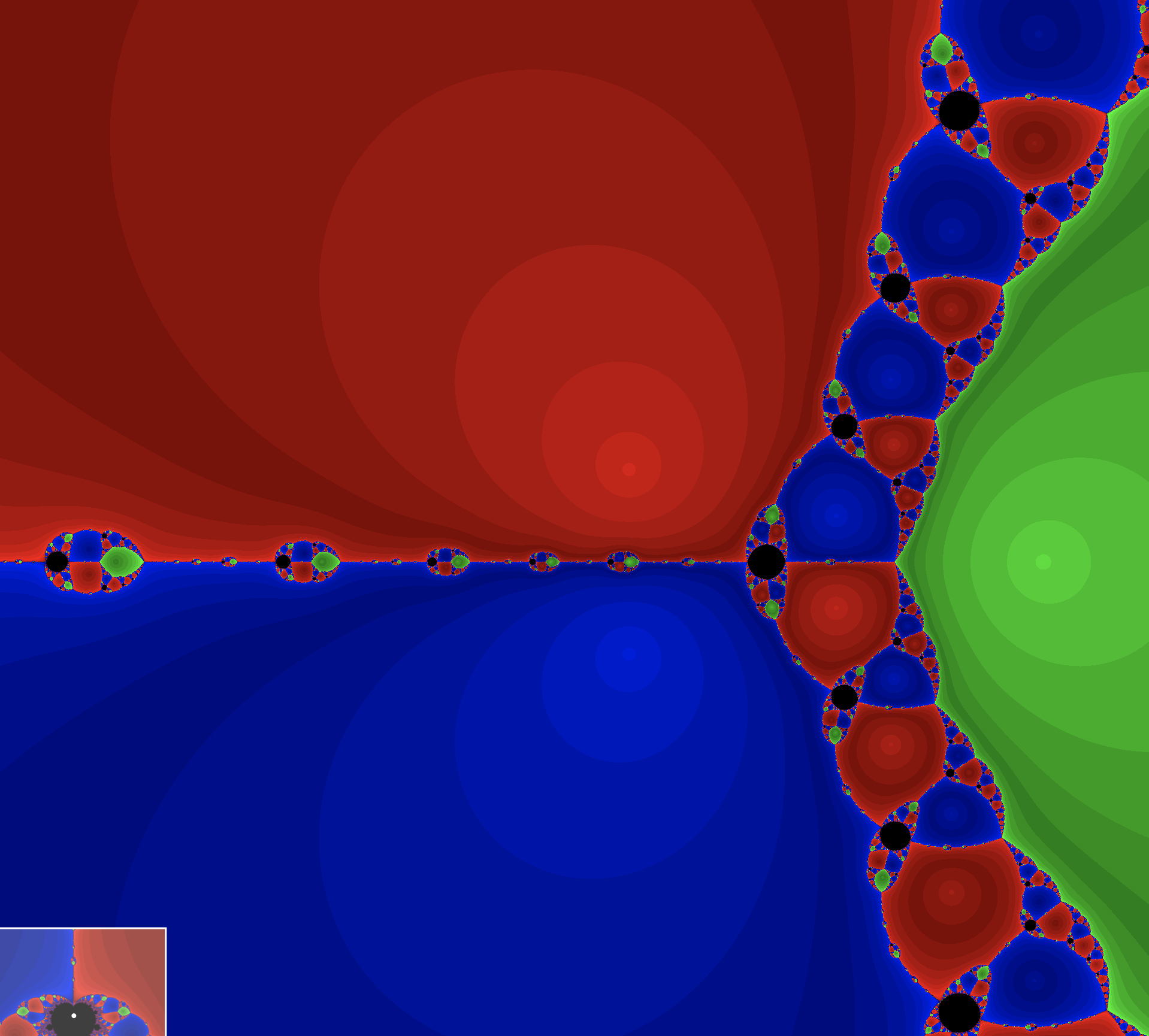}
  \qquad
  \includegraphics[height=.4\textwidth,trim=0 0 100 0,clip]{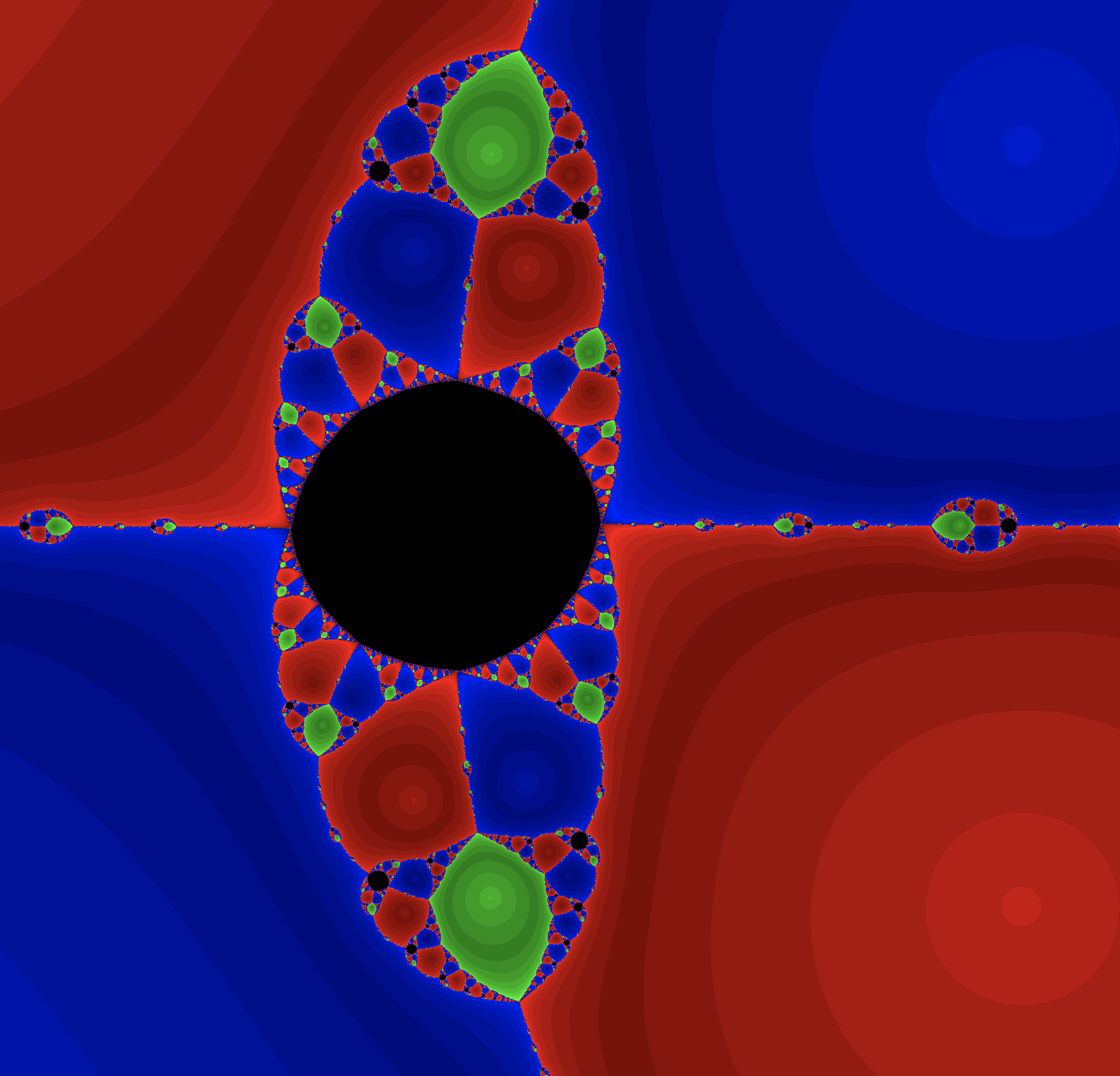}
  \caption{For every degree $d \ge 3$ and every period $m \ge 2$ there is a
           polynomial~$p$ of degree~$d$ so that $N_p$ has a periodic point of
           period~$m$ that attracts a neighborhood of each of its points.
           \textbf{Left:} The Newton dynamics plane for $p(Z) = Z^3 - 2Z + 2$, where
           $N_p(z) = z - \frac{z^3 - 2z + 2}{3z^2 - 2} = \frac{2z^3 - 2}{3z^2 - 2}$
           has an attracting $2$-cycle $0 \mapsto 1 \mapsto 0$; its basin is shown
           in black. \textbf{Right:} Detail near center.}
  \label{Fig:AttractingCycle}
\end{figure}

Newton's method for polynomials of degree~$1$ is trivial: the Newton map
is the constant map with value at the root. For degree~$2$, the dynamics is
very simple as well; we note this here for later use.

\begin{Lemma}[Newton's method for quadratic polynomials]
  \label{L:Newton-deg2}
  If $p$ is a polynomial of degree~$2$ with distinct roots,
  then $N_p$ is conformally conjugate to the squaring map $z \mapsto z^2$
  on the Riemann sphere. In particular, $N_p$ has periodic orbits
  of each exact period at least~$2$, none of which are attracting.
\end{Lemma}

\begin{proof}
  By Lemma~\ref{L:Newton-affine}, we can take $p(Z) = Z^2 - 1$. Then
  \[ N_p(z) = \frac{z^2 + 1}{2z} = T^{-1}\bigl(T(z)^2\bigr) \quad\text{with}\quad
     T(z) = \frac{z+1}{z-1} \,.
  \]
  Now fix~$n$ and let $\omega$ be a primitive $(2^n-1)$-th root
  of unity. Then $\omega$ has exact order~$n$ under the squaring
  map, so $T^{-1}(\omega)$ has exact order~$n$ under~$N_p$.
  The multiplier of~$\omega$ as a point of order~$n$ is~$2^n$,
  and this is the same as the multiplier of~$T^{-1}(\omega)$
  under~$N_p$.
\end{proof}

For completeness, we might note  that the Newton map for a quadratic
polynomial with a double root is conformally conjugate to $z \mapsto z/2$.


\subsection*{Positive results about Newton's method}

Meanwhile, there is a substantial body of knowledge about the global dynamics
of Newton's method, in stark contrast to the Weierstrass method.
Here we mention some of the relevant results.

For the Newton dynamics~$N_p$, any particular orbit may or may not converge
to a root. However, one can estimate that asymptotically at least a fraction
of $1/(2 \log 2) \approx 0.72$
of randomly chosen points in $\C$ will converge to some root
(see~\cite{HSS}*{Section~4}). More explicitly, for every degree~$d$ there
is a universal set~$S_d$ of
starting points that will find, for every polynomial~$p$ of degree~$d$,
normalized so that all roots are in the unit disk, all the roots of~$p$
under iteration of~$N_p$. This set is universal in the sense that it depends
only on~$d$, and it may have cardinality as low as $1.1 d (\log d)^2$~\cite{HSS}.
If one accepts probabilistic results, then $c d (\log\log d)^2$ starting
points are sufficient to find all roots with given probability,
where $c$ depends only on this probability~\cite{BLS}.
Upper bounds on the complexity of
Newton's method to find all roots with prescribed precision~$\eps$ were
established in~\cites{NewtonEfficient,BAS}; they can be as good as
$O(d^2 (\log d)^4 + d \log |\log\eps|)$, which is close to optimal when the
starting points are outside of a disk containing the roots.

In addition to these strong theoretical results, Newton's method has also
been used successfully in practice for finding all roots of polynomials of
degrees exceeding~$10^9$~\cite{NewtonRobin1,NewtonRobin2}, and it is
interesting to compare the experimental complexity between the Newton and
Ehrlich--Aberth methods; see~\cite{NewtonExperiments}: depending on the
efficiency how the polynomials can be evaluated, and how the roots are
located, one or the other method may be faster.

Finally, we might mention that there are several other complex one-dimensional
root finding iteration methods, including K\"onig's method;
see~\cite{XavierChristianRootFinders}. However, there is a theorem by
McMullen~\cite{McMullenRootFinding} that no one-dimensional root
finding method can be generally convergent.
It is natural to ask whether a similar result holds also for
root finding methods in several variables.


\section{The Weierstrass method} \label{Sec:Weierstrass}

The \emph{Weierstrass root finding method}, also known as the Durand--Kerner
me\-thod, tries to approximate all~$d$ roots  of a degree $d$ polynomial
simultaneously (unlike the Newton method, which approximates only one root
at a time). Recall that $\CP'_d$ is the space of
monic polynomials of degree~$d$. Let $p \in \CP'_d$. Then the
Weierstrass root finding method consists of iterating the
(partially defined) map $W_p \colon \C^d \to \C^d$, $\uz \mapsto \uz'$,
where the components~$z'_k$ of~$\uz'$ are given in terms of those of~$\uz$ by
\begin{equation} \label{Eq:WeierstrassIteration}
  z'_k = z_k - \frac{p(z_k)}{\prod_{j\neq k} (z_k-z_j)} \,.
\end{equation}
This map is defined for all $\uz \in \C^d \sm \Delta$, where $\Delta$ is
the ``big diagonal''
\[ \Delta = \{\uz \in \C^d : \text{$z_j = z_k$ for some $1 \le j < k \le d$}\} \,. \]

If $p$ is not necessarily monic, then $W_p$ is defined
to be the same as $W_{p/c}$, where $c$ is the leading coefficient of~$p$.
It is therefore sufficient to consider only monic polynomials.

The Weierstrass method converges on a non-empty open subset of~$\C^d$ to a vector
containing the $d$~roots in some order. It is well known that iteration of $W_p$ may land on $\Delta$ after any number of steps even when the starting point $\uz$ is not in $\Delta$. Moreover, even when an orbit is defined forever it may fail to converge to roots: for instance, when a polynomial is
real but its roots are not, then a vector of purely real initial points
cannot converge to non-real solutions; similar arguments apply in the presence
of other symmetries. More generally, different vectors of starting points
may converge to the roots in different order, and the respective domains
of convergence in~$\C^d$ must have non-empty boundaries on which convergence
cannot occur. The best possible outcome to hope for would be that convergence
to roots occurs on an open dense subset of~$\C^d$, ideally with complement
of measure zero.

Obviously, if $z_k$ is already a root, then the map has a fixed point in
the $k$-th coordinate; all roots already found stabilize in the approximation
vector (as long as they are all distinct).

One heuristic interpretation of the Weierstrass method is as follows. Each of the
$d$~component
variables ``thinks'' that all other roots have already been found and tries
to find its own value necessary to match the value of the polynomial at a
single point. To make this precise, write again $p(Z) = \prod_k (Z-\alpha_k)$.
Take a coordinate $k \in \{1, \dots, d\}$; if we assume that $z_j = \alpha_j$
for all $j \neq k$, then
\begin{equation} \label{Eq:WeierstrassHeuristics}
  p(z_k) = (z_k-\alpha_k) \prod_{j\neq k} (z_k-z_j) \,,
\end{equation}
and then the method simply ``finds'' the missing root~$\alpha_k$ as the only
unknown quantity in~\eqref{Eq:WeierstrassHeuristics} to make the equation fit.
This leads to the Weierstrass iteration formula~\eqref{Eq:WeierstrassIteration}.
For the Weierstrass method, all~$k$ variables make the same ``assumption''
and in general they are all wrong, but it turns out anyway that this leads to
a reasonable approximation of the root vectors, at least sufficiently close
to a true solution.

We will now show that the Weierstrass method can be interpreted as a
higher-dimensional Newton iteration. Consider the map
\[ F \colon \C^d \To \CP'_d \,, \qquad
            (z_1, \dots, z_d) \longmapsto \prod_{k=1}^d (Z-z_k) \,.
\]
Then the task of finding all the roots of~$p$ is equivalent to finding
some preimage of~$p$ under~$F$. To solve this problem, we can employ
Newton's method in~$d$ dimensions. This leads to the iteration
\begin{equation} \label{Eq:WasN}
  \uz \longmapsto \uz - (\TotalDiff{F}{\uz})^{-1} (F(\uz) - p) \,,
\end{equation}
which is defined on the set of~$\uz \in \C^d$ where $\TotalDiff{F}{\uz}$
is invertible, which is the case if and only if $\uz \notin \Delta$.
(The ``if'' direction follows from the proof below; the
``only if'' direction is easy.)

\begin{Lemma}[Weierstrass method as higher-dimensional Newton]
  The map given by~\eqref{Eq:WasN} is~$W_p$.
\end{Lemma}

\begin{proof}
  First note that the partial derivative of~$F$
  with respect to the $k$-th coordinate~$z_k$ is
  \[ \frac{\partial F}{\partial z_k}(\uz) = -\prod_{j \neq k} (Z - z_j) \,, \]
  where the expression on the right is a polynomial of degree less than $d$;
  we identify the space of such polynomials with~$\C^d$.
  If we denote the right hand side of~\eqref{Eq:WasN} by~$\uz'$, we can
  write \eqref{Eq:WasN} in the form
  \begin{equation} \label{Eq:WeierstrassGood}
    \TotalDiff{F}{\uz} (\uz' - \uz) = p - F(\uz) \,.
  \end{equation}
  Written out, this gives
  \begin{equation} \label{Eq:Wgoodexplicit}
    \sum_{k=1}^d (z'_k - z_k) \prod_{j \neq k} (Z - z_j) = \prod_{k=1}^d (Z - z_k) - p \,.
  \end{equation}
  If we assume that the entries of~$\uz$ are distinct and,
  separately for each $m \in \{1, \ldots, d\}$, we set $Z \leftarrow z_m$,
  the product on the right and most products
  on the left vanish and the remaining equation gives~\eqref{Eq:WeierstrassIteration}
  (with $m$ in place of~$k$).
\end{proof}

The following local convergence result is well known.

\begin{Lemma}[Local convergence of the Weierstrass method]
  \label{Lem:LocalConvergenceWeierstrass}
  Every vector consisting of the $d$~roots of~$p$ has a neighborhood in~$\C^d$
  on which the Weierstrass method converges to this solution vector.
  This convergence is quadratic when $p$ has no multiple roots.
\end{Lemma}

\begin{proof}
  This follows from the fact that $W_p$ is Newton's method applied
  to $F(\uz) - p$.
\end{proof}


\subsection{Properties of the Weierstrass method}

We state some elementary and well known properties of~$W_p$ that will be
important to us.

\begin{Lemma}[Simple properties of the Weierstrass method] \label{L:Weierstrass-1} \strut
  \begin{enumerate}[\upshape(1)]
    \item Let $p \in \CP'_d$ and $T \in \Aff(\C)$. Then $W_{Tp}$
          is is conformally conjugate to~$W_p$ by~$T$, i.e.,
          $W_{Tp} = T \circ W_p \circ T^{-1}$, where the action of~$T$
          on~$\C^d$ is component-wise.
    \item For each $p \in \CP'_d$,
          $W_p$ is equivariant with respect to the natural action
          of the symmetric group$~S_d$ on on~$\C^d$ by permuting
          the coordinates: if $\sigma \in S_d$, then $W_p(\sigma \uz) = \sigma W_p(\uz)$.
  \end{enumerate}
\end{Lemma}

\begin{proof} \strut
  \begin{enumerate}[\upshape(1)]
    \item Writing $p = \prod_{k=1}^d (Z - \alpha_k)$ in~\eqref{Eq:Wgoodexplicit},
          we see that the relation is unchanged when we replace $\alpha_k$, $z_k$,
          $z'_k$ and~$Z$ by their images under~$T$.
          Undoing the transformation on~$Z$ then gives a valid equation between
          polynomials, which is equivalent to $W_{Tp}(T \uz) = T W_p(\uz)$,
          or $W_{Tp} = T \circ W_p \circ T^{-1}$, where the action of affine
          transformations on~$\C^d$ is coordinate-wise.
    \item This is clear.
    \qedhere
  \end{enumerate}
\end{proof}

By the first property we can use the same parameter space~$\CP_d$ for
the Weierstrass iteration on polynomials of degree~$d$ as we did for
Newton's method.

Equation~\eqref{Eq:Wgoodexplicit} leads to a simple proof of the following
useful property.

\begin{Lemma}[Invariant hyperplane] \label{L:Weierstrass-2}
  Let $p = Z^d - a Z^{d-1} + \ldots$. Then the sum of the entries
  of~$W_p(\uz)$ is~$a$, for all $\uz \in \C^d \sm \Delta$.
\end{Lemma}

\begin{proof}
  Comparing coefficients of~$Z^{d-1}$ in~\eqref{Eq:Wgoodexplicit},
  we see that
  \[ \sum_{k=1}^d (z'_k - z_k) = -\sum_{k=1}^d z_k + a\,, \]
  which gives the claim.
\end{proof}

This means that the dynamics is effectively only $(d-1)$-dimensional and
takes place on the hyperplane $z_1 + \ldots + z_d = a$.
As mentioned earlier, we can restrict to centered polynomials,
i.e., $a = 0$.

\begin{Lemma}[Degree reduction if root is present] \label{L:Weierstrass-3}
  Fix $k \in  \{1,\ldots,d\}$.
  If $z_k$ is a root of~$p$ and $\uz \in \C^d \sm \Delta$,
  then $W_p(\uz)_k = z_k$, and the dynamics on
  the remaining entries is that of the Weierstrass method for~$p(Z)/(Z-z_k)$.
\end{Lemma}

\begin{proof}
  Clear from the definition.
\end{proof}

\begin{Lemma}[Weierstrass in degree~$2$ is Newton] \label{L:Weierstrass-4}
  If $p$ has degree~$2$, then the dynamics of~$W_p$ reduces to Newton's
  method for~$p$.
  In particular, for $p$ with distinct roots, $W_p$ restricted
  to the invariant hyperplane (which is a line in this case)
  is conjugate to the squaring map $z \mapsto z^2$,
  which has no attracting cycles that are not fixed points.
\end{Lemma}

\begin{proof}
  By Lemma~\ref{L:Weierstrass-1}, we can assume that $p(Z) = Z^2 - 1$
  if $p$ has distinct roots.
  By Lemma~\ref{L:Weierstrass-2}, all iterates after the initial vector will
  have the form $(z, -z)$. It is then easy to check that $W_p(z, -z) = (w, -w)$
  with $w = z - (z^2-1)/(2z) = N_p(z)$. The last claim follows from
  Lemma~\ref{L:Newton-deg2}.

  If $p$ has a double root, then $N_p$ and $W_p$ are conjugate
  to $N_{Z^2}$ and~$W_{Z^2}$, respectively; again, $W_{Z^2}$
  agrees with~$N_{Z^2}$ when restricted to the invariant line.
\end{proof}

When $p$ is linear, then $N_p$ and~$W_p$ both find the unique
root immediately by definition. So this lemma tells
us that interesting behavior in the Weierstrass method can occur only when
$d \ge 3$.

When looking for periodic orbits under~$W_p$, Lemma~\ref{L:Weierstrass-3}
tells us that we can assume that no entry of~$\uz$ is a root of~$p$,
since otherwise we can reduce to a case of lower degree.
However, this very observation allows us to promote counterexamples
of low degrees to higher degrees. To do this, we need the following lemma.

\begin{Lemma}[Lifting to higher degrees] \label{L:lift}
  Let $p$ be a monic polynomial of degree~$d$ and let $\alpha \in \C$.
  Set $\tilde{p}(Z) = (Z-\alpha) p(Z)$.
  \begin{enumerate}[\upshape(1)]
    \item For a point $\uz = (z_1, \ldots, z_d, \alpha)$
          with pairwise distinct entries,
          the Jacobi matrix $\TotalDiff{W_{\tilde{p}}}{\uz}$ has the form
          \[ \TotalDiff{W_{\tilde{p}}}{\uz}
              = \begin{pmatrix}
                          \TotalDiff{W_p}{\uz'} & *_{d \times 1} \\
                          0_{1 \times d} & \lambda
                \end{pmatrix}
          \]
          with $\uz' = (z_1, \ldots, z_d)$ and
          \[ \lambda = 1 - \frac{p(\alpha)}{\prod_{j=1}^d (\alpha - z_j)} \,. \]
    \item If $q \in \C^d$ is a periodic point of~$W_p$ of period~$n$
          such that all eigenvalues of~$\TotalDiff{W_p^{\circ n}}{q}$
          have absolute values strictly less than~$1$, then for $|\alpha|$
          sufficiently large, $\tilde{q}: = (q, \alpha) \in \C^{d+1}$ is a periodic
          point of~$W_{\tilde{p}}$ of period~$n$ such that
          all eigenvalues of~$\TotalDiff{W_{\tilde{p}}^{\circ n}}{\tilde{q}}$
          have absolute values strictly less than~$1$.
  \end{enumerate}
\end{Lemma}

\begin{proof}
  The first claim results from an easy computation.

  Now assume that $q \in \C^d$ is a periodic point of~$W_p$
  of period~$n$. By Lemma~\ref{L:Weierstrass-3}, $\tilde{q} = (q, \alpha)$
  is a periodic point of~$W_{\tilde{p}}$ of period~$n$.
  We obtain an analogous formula relating the derivatives
  of $W_p^{\circ n}$  at~$q$ and $W_{\tilde{p}}^{\circ n}$ at~$\tilde{q}$,
  with the product $\lambda_0 \cdots \lambda_{n-1}$ replacing~$\lambda$,
  where $\lambda_m$ arises from $(W_p^{\circ m}(q), \alpha)$.
  In particular, the eigenvalues
  of~$\TotalDiff{W_{\tilde{p}}^{\circ n}}{\tilde{q}}$ are those
  of~$\TotalDiff{W_p^{\circ n}}{q}$ together with~$\lambda_0 \cdots \lambda_{n-1}$.
  As $|\alpha| \to \infty$, we see that $\lambda_m \to 0$ for all $0 \le m < n$,
  and the claim follows.
\end{proof}


\subsection{The dynamics of $W_{Z^3}$} \label{Sec:dynZ3}

In this section we prove some results on the dynamics of the Weierstrass
iteration in the simple case when $p(Z) = Z^3$.
By Lemma~\ref{L:Weierstrass-2}, we can restrict consideration to the
hyperplane $H = \{z_1 + z_2 + z_3 = 0\}$.
We will show that all
starting points in~$H$ outside a set of measure zero converge to the unique
root vector~$(0,0,0)$, but that there are uncountably many orbits
that converge to infinity.

From Lemma~\ref{L:Weierstrass-1}~(1)
and since the unique root~$0$ of~$Z^3$ is invariant under scaling,
it follows that $W_{Z^3}(\lambda \uz) = \lambda W_{Z^3}(\uz)$, so $W_{Z^3}$
induces a rational map $\varphi \colon \BP H \to \BP H$, where
$\BP H \simeq \BP^1$ is the complex projective line obtained by considering
the nonzero points of~$H$ up to scaling.

Writing a nonzero point in~$H$ up to scaling in the form $(1, z, -1-z)$
(the missing scalar multiples of~$(0, 1, -1)$ correspond to the limit
case $z = \infty$), we find that
\begin{equation}
  W_{Z^3}(1, z, -1-z) = s(z) (1, \varphi(z), -1-\varphi(z))
\end{equation}
with
\begin{equation}
  s(z) = \frac{1-z-z^2}{2-z-z^2} \qquad\text{and}\qquad
  \varphi(z) = \frac{z (2+z) (1+z-z^2)}{(1+2z) (1-z-z^2)} \,.
\end{equation}

We can say something about the dynamics of~$\varphi$.

\begin{Lemma}[Dynamics of~$\varphi$] \label{L:dynphi}
  The map~$\varphi$ has two attracting fixed points at $\omega$ and~$\omega^2$,
  where $\omega = e^{2 \pi i/3}$ is a primitive cube root of unity. Let $z \in \C$.
  If $\Im(z) > 0$,
  then $\varphi^{\circ n}(z)$ converges to~$\omega$ as $n \to \infty$,
  and if $\Im(z) < 0$, then $\varphi^{\circ n}(z)$
  converges to~$\omega^2$ as $n \to \infty$. The real line is
  forward and backward invariant under~$\varphi$.
\end{Lemma}

\begin{proof}
  Conjugating $\varphi$ by the M\"obius transformation
  $z \mapsto (\omega^2 z - \omega^2)/(z - \omega)$, we obtain
  \[ f(z) = z \frac{2z^3 + 1}{z^3 + 2} \,. \]
  This map~$f$ is the product of~$z$ with the composition of $z\mapsto z^3$
  by $z \mapsto (2z+1)/(z+2)$; the latter is an automorphism of the open unit disk
  (and also of the complement of the closed unit disk in the Riemann sphere).
  Therefore, $|f(z)| < |z|$ when $0<|z| < 1$ and $|f(z)| > |z|$ for $|z| > 1$
  (in other words, $f$ is a Blaschke product with a fixed point at $z = 0$).
  This implies that the open unit disk is attracted to the fixed point~$0$
  of~$f$, while the complement of the closed unit disk is attracted to~$\infty$;
  the unit circle is forward and backward invariant
  (and maps to itself as a covering map with degree~$4$).
  Translating back to~$\varphi$, this gives the result.
\end{proof}

From this, we can deduce the following statement on the global dynamics
of~$W_{Z^3}$.

\begin{Theorem}[Convergence of $W_{Z^3}$] \label{T:convZ3}
  If $\uz \in H$ is not a scalar multiple of a vector with real entries,
  then $W_{Z^3}^{\circ n}(\uz)$ converges to the zero vector. The convergence
  is linear with rate of convergence~$2/3$.
\end{Theorem}

Note that the rate of convergence is the same as that for~$N_{Z^3}$.

\begin{proof}
  Let $\uz \in H$ be such that $\uz$ is not a scalar multiple of
  a vector with real entries. In particular, $\uz$ is not the zero vector.
  By symmetry, we can assume that the first entry is nonzero; then
  $\uz = z_1 (1, z, -1-z)$ with $z = z_2/z_1 \in \C \sm \R$. We then have that
  $\varphi^{\circ n}(z) \neq \infty$ for all $n \ge 0$
  by Lemma~\ref{L:dynphi}, and
  \[ W_{Z^3}^{\circ n}(\uz)
      = \prod_{k=0}^{n-1} s(\varphi^{\circ k}(z))
           \cdot (1, \varphi^{\circ n}(z), -1-\varphi^{\circ n}(z)) \,.
  \]
  By Lemma~\ref{L:dynphi} again, $\varphi^{\circ n}(z)$ converges to $\omega$
  or~$\omega^2$. Since $s(\omega) = s(\omega^2) = 2/3$, the factor in front
  will linearly converge to zero with rate of convergence~$2/3$, whereas the
  vector will converge to~$(1, \omega, \omega^2)$ (if $\Im(z) > 0$)
  or to~$(1, \omega^2, \omega)$ (if $\Im(z) < 0$).
\end{proof}

We have seen that orbits of starting points in~$H$ that are not scalar multiples
of real vectors converge to zero, whereas there are real vectors in~$H$ whose
orbit tends to infinity; see Section~\ref{Sec:escape} below.
There are also starting points whose orbits cease to
be defined after finitely many steps; this occurs if and only if some iterate
is a multiple of~$(1,1-2)$ or one of its permutations. On the other hand, there
are many real starting points in~$H$ whose orbits tend to zero (one example
is obtained by replacing $\alpha$ with~$-\alpha$ in the proof of
Theorem~\ref{T:escZ3} below). In fact, we expect that almost all real
starting points have this property.


\subsection{Escaping points} \label{Sec:escape}

In this section we prove Theorem~\ref{MT2}: the Weierstrass iteration~$W_p$
has escaping points for all polynomials of degree $d \ge 3$ with distinct roots.

We first continue our study of the cubic case, $d = 3$. As observed earlier,
we can always assume that our polynomial~$p$ is centered, i.e.,
has the form $p = Z^3 + a Z + b$. Then the image of~$W_p$ is contained in the
plane $H = \{z_1 + z_2 + z_3 = 0\}$, so it is sufficient to consider the
induced map $H \to H$. We identify~$H$ with~$\C^2$ by projecting to the
first two coordinates. We can then extend~$W_p$ to a rational map $\BP^2 \to \BP^2$,
which is given by the following triple of quartic polynomials, as a simple
computation shows.
\begin{align}
  (z_1 : z_2 : z_0) \longmapsto
      \bigl(&(z_1 + 2z_2)(z_1^3 + a z_1 z_0^2 + b z_0^3 - z_1 z_2(z_1 + z_2)) \nonumber \\
            &{} : (2z_1 + z_2)(-z_2^3 - a z_2 z_0^2 - b z_0^3 + z_1 z_2(z_1 + z_2)) \label{Eq:ratmapP2}\\
            &{} : z_0 (z_1 - z_2)(z_1 + 2z_2)(2z_1 + z_2)\bigr) \nonumber
\end{align}
Here the line at infinity is given by $z_0 = 0$; it is forward invariant,
and the induced dynamics on this projective line is given by the rational
map~$\varphi$ from Section~\ref{Sec:dynZ3}.

\begin{Theorem}[Escaping orbits for cubic polynomials] \label{T:escZ3}
  For every cubic polynomial~$p$,
  there are starting points $\uz \in \C^3$ such that the iteration
  sequence $(W_p^{\circ n}(\uz))$ exists for all times and converges
  component-wise to infinity. The set of escaping points contains
  a holomorphic curve.
\end{Theorem}

\begin{proof}
  Let
  \[ \alpha = -\sqrt{\frac{5 + \sqrt{21}}{2}} \,. \]
  Then one can check that the point $q_0 = (1 : \alpha : 0)$ on the line at infinity
  is $2$-periodic for the extension of~$W_p$ to~$\BP^2$.
  We consider $q_0$ as a fixed point of the second iterate of this extension.
  Its multiplier matrix has eigenvalues
  \begin{align*}
    12 \alpha^3 - 72 \alpha + 43 = 43 + 12 \sqrt{7} > 1 \quad\text{and}\quad
    \frac{-\alpha^3 + 6 \alpha + 4}{2} = 2 - \tfrac{1}{2} \sqrt{7} \in (0, 1) \,.
  \end{align*}
  The eigenspace for the first of these eigenvalues is tangential to the line
  at infinity, whereas the eigenspace for the second eigenvalue points away
  from the line. So the point~$q_0$ has a stable manifold
  (see~\cite{Palis-deMelo}*{Ch.~2, Section~6} for the general theory) that meets the
  (complex) line at infinity locally only at~$q_0$ and is a holomorphic
  curve by~\cite{Hubbard2005}*{Cor.~8}.
  In particular, all points $q \in H$ that lie on the stable manifold and
  are sufficiently close to~$q_0$ will converge in~$\BP^2$ to the $2$-cycle
  that $q_0$ is part of. Since the points of this $2$-cycle are on the line
  at infinity (and different from $(1 : 0 : 0)$, $(0 : 1 : 0)$, $(1 : -1 : 0)$,
  which are the points corresponding to the lines $z_1 = 0$, $z_2 = 0$
  and $z_3 = -z_1 -z_2 = 0$), the claim follows.
\end{proof}

\begin{Remark}
\label{Rem:CubicEscaping}
  When $p = Z^3$, the stable manifold of~$q_0$ is the complex line
  joining it to~$(0:0:1)$. So in this case, every scalar multiple of
  $(1, \alpha, -1-\alpha) \in H$ escapes to infinity.
\end{Remark}

Now Theorem~\ref{MT2} follows from Theorem~\ref{T:escZ3} in the following
way. Write $p = p_1 p_2$ with $p_1$ of degree~$3$ and $p_2$ with simple
roots. By Theorem~\ref{T:escZ3}
there is a vector $q_1 \in \C^3$ that escapes to infinity under~$W_{p_1}$.
Now set $q = (q_1, q_2)$, where $q_2 \in \C^{d-3}$ has the roots of~$p_2$
(in some order) as entries. Then iterating~$W_p$ on~$q$ has the effect of
fixing the last $d-3$ coordinates, whereas the effect on the first three
is that of~$W_{p_1}$; see Lemma~\ref{L:Weierstrass-3}. In particular, the
first three coordinates of the vectors in the orbit of~$q$ under~$W_p$
tend to infinity. Note that this result covers a slightly larger set of
polynomials than those with simple roots: the cubic factor~$p_1$ is
arbitrary, so $p$ can have a multiple root of order at most~$4$ or
two double roots.

Taking iterated preimages under~$W_p$ of the curve to infinity whose
existence we have shown in Theorem~\ref{T:escZ3} above, we obtain countably
infinitely many (complex) curves to infinity full of escaping points.
Here we restrict to iterated preimage curves ending in an iterated preimage
of the point~$q_0$ (notation as in the proof above) that is on the
line at infinity. Two of the immediate preimage curves end at the origin,
which is a point of indeterminacy for the rational map~\eqref{Eq:ratmapP2}
induced by~$W_p$. There are very likely other escaping points,
but we expect the set of escaping points to be
of measure zero within~$H$.


\section{Algebraic description of periodic orbits} \label{Sec:PeriodicPoints}

Since we will be using methods from Computer Algebra to obtain a proof
of the Theorem~\ref{MT1}, we now discuss how we can describe the periodic
points of~$W_p$ of any given period algebraically. We begin with a
description of~$W_p$ itself.


\subsection{Algebraic description of $W_p$}

For the purpose of studying periodic orbits under~$W_p$ algebraically
as $p$~varies, equation~\eqref{Eq:Wgoodexplicit} is preferable
to~\eqref{Eq:WeierstrassIteration}, since it is a polynomial equation
involving the entries of $\uz$ and~$\uz'$ and the coefficients of~$p$,
rather than an equation involving rational functions. The following
result shows that we do not get extraneous solutions by doing so,
in the sense that all solutions we find that involve points in~$\Delta$
arise as degenerations of ``honest'' solutions living outside~$\Delta$.

\begin{Proposition}[Polynomial equation describing iteration]
  \label{Prop:W-good}
  Fix $p \in \CP'_d$. The algebraic variety in~$\C^d \times \C^d$
  described by equation~\eqref{Eq:Wgoodexplicit} is the Zariski
  closure of the graph of~$W_p$ (which is contained
  in $(\C^d \sm \Delta) \times \C^d$).
\end{Proposition}

\begin{proof}
  Let $V_p$ denote the variety in question.
  Equation~\eqref{Eq:Wgoodexplicit} corresponds to $d$~equations
  in the $2d$~coordinates of $\uz$ and~$\uz'$, so each irreducible
  component of~$V_p$ must have dimension at least~$d$. We have to
  show that no irreducible component is contained in $\Delta \times \C^d$.
  We do this by showing that $\dim (V_p \cap (\Delta \times \C^d)) < d$.

  Assume that $\uz \in \Delta$. We first consider the simplest
  case that $z_1 = z_2$, but $z_2, \ldots, z_d$ are distinct.
  Substituting $Z \leftarrow z_1$ in~\eqref{Eq:Wgoodexplicit}, we obtain
  that $p(z_1) = 0$, so that $z_1$ must be a root of~$p$.
  The subset of~$\Delta$ consisting of~$\uz$ with this property
  has dimension~$d-2$. Substituting $Z \leftarrow z_k$ with $k \ge 3$,
  we see that $z'_k$ is uniquely determined by~$\uz$ (it is still
  given by~\eqref{Eq:WeierstrassIteration}). On the other hand,
  taking the derivative with respect to~$Z$ on both sides and then
  substituting $Z \leftarrow z_1$, we see that $z'_1 + z'_2$ is uniquely
  determined, so the fiber above~$\uz$ of the projection of~$V_p$
  to the first factor has dimension~$1$. So the part of
  $V_p \cap (\Delta \times \C^d)$ lying above points~$\uz$ with only
  one double entry has dimension~$d-1$.

  In general, we see by similar
  considerations (taking higher derivatives as necessary) that
  when $\uz$ has entries of multiplicities $m_1, \ldots, m_l$
  (with $m_1 + \ldots + m_l = d$ and some $m_j \ge 2$),
  then these entries must be roots
  of~$p$ of multiplicities (at least) $m_1 - 1, \ldots, m_l - 1$,
  and the fiber of~$V_p$ above~$\uz$ is a linear space of dimension
  $(m_1 - 1) + \ldots + (m_l - 1) = d - l$. On the other hand, the
  set of~$\uz$ of this type has dimension $\#\{j : m_j = 1\} < l$,
  so the dimension of the corresponding subset of~$V_p$ is~$< d$.

  So we have seen that $V_p \cap (\Delta \times \C^d)$ is a finite
  union of algebraic sets of dimension $< d$; therefore it cannot
  contain an irreducible component of~$V_p$.
\end{proof}

\begin{Remark} \label{Rem:DegreeWeierstrassGood}
  As in the proof above, we will usually think of~\eqref{Eq:Wgoodexplicit}
  as a system of $d$~equations that are obtained by comparing the coefficients
  of the various powers of~$Z$ on both sides.
  Note that the equation for the coefficient of~$Z^j$ is of degree $d - j$
  in $z_1, \dots, z_d, z_1', \dots, z_d'$. So the total system has degree~$d!$.
\end{Remark}


\subsection{Periodic points} \label{SS:perpts}

We use equation~\eqref{Eq:Wgoodexplicit} to obtain a system of equations
representing periodic points. Fix the degree~$d$ and the period~$n$.
We consider $nd$~variables, grouped into $n$~vectors
$\uz^{(k)} = (z_1^{(k)}, \ldots, z_d^{(k)})$, for $0 \le k < n$,
which we think of as representing an $n$-cycle
$\uz^{(0)}$, $\uz^{(1)} = W_p(\uz^{(0)})$, \dots, $\uz^{(n-1)} = W_p(\uz^{(n-2)})$,
$\uz^{(0)} = W_p(\uz^{(n-1)})$. We therefore define the scheme
$\CP'_d(n) \subset \CP'_d \times \C^{nd}$ by collecting the equations
arising from comparing coefficients on both sides of~\eqref{Eq:Wgoodexplicit},
where we replace $(\uz, \uz')$ successively by $(\uz^{(0)}, \uz^{(1)})$,
$(\uz^{(1)}, \uz^{(2)})$, \dots, $(\uz^{(n-1)}, \uz^{(0)})$; $p$ runs through
the monic degree~$d$ polynomials in~$\CP'_d$. This encodes that
$\uz^{(0)} \mapsto \uz^{(1)} \mapsto \ldots \mapsto \uz^{(n-1)} \mapsto \uz^{(0)}$
under~$W_p$. We then take $\CP_d(n)$ to
be the quotient of~$\CP'_d(n)$ by the group of affine transformations on~$\C$,
acting via
\[ T \cdot (p, z^{(0)}_1, \ldots, z^{(n-1)}_d) = (Tp, T z^{(0)}_1, \ldots, T z^{(n-1)}_d) \,. \]

We expect the fibers of the projection $\CP_d(n) \to \CP_d$ to be finite,
i.e., that for each polynomial~$p$, there are only finitely many points
of period~$n$ under~$W_p$. The following lemma gives a criterion for when
this is the case.

\begin{Lemma}[Criterion for finiteness of $n$-periodic points] \label{L:fincrit}
  Let $\CP_d^{(0)}(n) \subset \C^{nd}$ be the fiber of~$\CP'_d(n)$ above~$p = Z^d$.
  The projection $\CP'_d(n) \to \CP'_d$ is finite if and only if $\CP_d^{(0)}(n) = \{0\}$.
\end{Lemma}

\begin{proof}
  We first note that since the unique root~$0$ of~$Z^d$ is fixed by scaling,
  the same is true for~$\CP_d^{(0)}(n)$ under simultaneous scaling of the
  coordinates. So $\CP_d^{(0)}(n) = \{0\}$ is equivalent to $\CP_d^{(0)}(n)$
  being zero-dimensional. In particular, if $\CP_d^{(0)}(n) \neq \{0\}$,
  then the projection is not finite, since the fiber above~$Z^d$ has positive
  dimension. This proves one direction of the claimed equivalence.

  Now assume that the projection is not finite, so there is some $p \in \CP_d$
  such that the fiber~$\CP_d^{(p)}(n)$ above~$p$ has positive dimension.
  Let $\bar{\CP}_d^{(p)}(n) \subset \BP^{nd}$
  denote the projective scheme obtained by homogenizing the equations
  defining~$\CP'_d(n)$ and specializing to~$p$. Then $\CP_d^{(p)}(n)$
  meets the hyperplane at infinity of~$\BP^{nd}$. But the intersection
  of~$\CP_d^{(p)}(n)$ with the hyperplane at infinity is exactly
  the image of~$\CP_d^{(0)}(n)$ under the projection $\C^{nd} \sm \{0\} \to \BP^{nd-1}$.
  So this image is non-empty, which implies that $\CP_d^{(0)}(n)$ contains
  non-zero points. This shows the other direction.
\end{proof}

We can test the condition ``$\CP_d^{(0)}(n) = \{0\}$'' with a Computer
Algebra System by setting up the ideal that is generated by the
equations defining~$\CP_d^{(0)}(n)$, together with $z^{(0)}_1 - 1$
(for symmetry reasons,
if there is some nonzero point, then there is one with $z^{(0)}_1 \neq 0$,
and by scaling, we can assume that $z^{(0)}_1 = 1$). Then we compute
a Groebner basis for this ideal. The condition is satisfied if and only if
this Groebner basis contains~$1$. We did this for $d = 3$ and small
values of~$n$.

\begin{Lemma}[Finiteness of $n$-periodic points] \label{L:perptsfinite}
  For every cubic polynomial~$p$ with at least two distinct roots,
  there are only finitely many points of period~$n \le 8$ under~$W_p$.
  For cubic polynomials with a triple root, the statement holds
  for all $n \le 8$ except $n = 6$.
\end{Lemma}

\begin{proof}
  The claim follows for $n \in \{1,2,3,4,5,7,8\}$ from Lemma~\ref{L:fincrit}
  and a computation as described above. For $n = 6$, we find that
  $\CP_3^{(0)}(6)$ consists of six lines through the origin (plus
  the origin with high multiplicity). These six lines correspond
  to $6$-cycles of rotation type (see Section~\ref{Sec:Cubic} below
  for the definition). By an explicit computation (see also
  Proposition~\ref{Prop:6r}), we check that the fiber above
  any polynomial with at least two distinct roots
  of the scheme describing $6$-periodic points of rotation type is finite.
  For the remaining components of~$\CP'_3(6)$, we find that the corresponding
  part of~$\CP_3^{(0)}(6)$ has the origin as its only point; we can
  then conclude as in the proof of Lemma~\ref{L:fincrit} that
  there are only finitely many $6$-periodic points not of rotation
  type for all cubic polynomials.
\end{proof}

\begin{Remark}
  We expect that for cubic polynomials without a triple root,
  the statement of Lemma~\ref{L:perptsfinite} holds for all~$n$.
  For cubic polynomials with a triple root, we expect that the
  $6$-periodic points of rotation type are the only exceptions, i.e.,
  that there are no points of exact order~$n \ge 2$ except the
  $6$-periodic points of rotation type described in the proof above.

  We do not venture to formulate a conjecture for polynomials
  of degrees higher than~$3$. We did verify the criterion of
  Lemma~\ref{L:fincrit} also for $d = 4$ and $n = 1, 2, 3$, however;
  beyond that, the computations become infeasible.
\end{Remark}

There is a simple argument that shows that periodic points of
any order always exist.

\begin{Lemma}[Existence of periodic points]
  Fix a monic polynomial~$p$ of degree $d \ge 2$ with distinct roots.
  Then $W_p$ has periodic points of all periods $n \ge 1$.
\end{Lemma}

\begin{proof}
  For $n = 1$, all the vectors consisting of the roots of~$p$ in some
  order are fixed points. So we fix now some $n \ge 2$. Write
  $p(Z) = \prod_{j=1}^d (Z - \alpha_j)$.
  Let $\omega$ be a primitive $(2^n-1)$-th root of unity.
  Then $\omega$ has exact period~$n$ under the squaring map $z \mapsto z^2$, so
  by Lemma~\ref{L:Weierstrass-4}, there is a point~$(z_1, z_2)$ of exact order~$n$
  for the Weierstrass map associated to~$(Z-\alpha_1) (Z-\alpha_2)$.
  By Lemma~\ref{L:Weierstrass-3}, the point
  \[ \uz = (z_1, z_2, \alpha_3, \ldots, \alpha_d) \]
  then has exact period~$n$ under~$W_p$.
\end{proof}

One might ask whether there are always periodic points of all
periods that do not fix any coordinate (or even, for which all
coordinates have the same period~$n$).

We are interested in \emph{attracting periodic points}, i.e., points
$q \in \C^{d}$ with the property that there is a period $n \ge 2$ and a
neighborhood~$U$ of~$q$ in~$\C^d$ so that $W_p^{\circ mn}(z) \to q$
as $m \to \infty$ for all $z \in U$. Consider the linearization $\TotalDiff{W_p^{\circ n}}{q}$
of the first return map at the point~$q$. We call this the
\emph{multiplier matrix} of~$q$. Local fixed point theory relates
the topological property of being attracting to an algebraic property of
this matrix, as in the following statement, which is a consequence
of the fact that a differentiable map is locally well-approximated
by its derivative.

\begin{Lemma}[Attracting fixed point]
  \label{Lem:AttractingFixedPoint}
  The fixed point~$q$ of a differentiable map \hbox{$W \colon \C^d \to \C^d$} is
  attracting if all eigenvalues of~$\TotalDiff{W}{q}$ have absolute values
  strictly less than~$1$. It cannot be attracting unless all eigenvalues
  have absolute values at most~$1$.
\end{Lemma}

In the context of points of period~$n$, we consider $W = W_p^{\circ n}$.
The lemma then tells us that $q$ can only be attracting when all
eigenvalues of its multiplier matrix have absolute value at most~$1$.
Equivalently, the characteristic polynomial of the multiplier matrix
has all its roots in the closed complex unit disk. The set of monic polynomials
of degree~$d$ with this property forms a compact subset~$\CA_d$ of~$\CP'_d$.

In the following, we
will always assume that we pick a representative in the affine equivalence
class of the polynomial in question that is \emph{centered},
i.e., with vanishing sum of roots.
Then the dynamics of~$W_p$ takes place in the linear hyperplane~$H$ given by
$z_1 + \ldots + z_d = 0$, and we get $\CP_d(n) \subset \CP_d \times H^n$.
We can identify $\CP_d(n)$ with its image in~$\CP_d \times H$
obtained by projection to the first two factors,
$(p, \uz^{(0)}, \ldots, \uz^{(n-1)}) \mapsto (p, \uz^{(0)})$.
Then the points
of~$\CP_d(n)$ are represented by pairs $(p, q)$, where $p$ is a centered
polynomial and $q \in H$ satisfies $W_p^{\circ n}(q) = q$.
Since we restrict to~$H$, the multiplier matrix of any periodic
point~$q$ is of size $\dim H = d-1$.

To study whether the $n$-cycles parameterized by~$\CP_d(n)$
can be attracting, we would like to associate to each such point $(p, q)$
the $d-1$~eigenvalues of the multiplier matrix of~$q$ (the eigenvalues do not
change under affine conjugation, so this gives a well-defined map).
However, there is no natural order on these
eigenvalues. To capture them as an unordered $(d-1)$-tuple, we express the
eigenvalues instead through their elementary symmetric functions and hence
through the characteristic polynomial of the multiplier matrix.
In this way, we obtain an algebraic morphism
(and therefore a holomorphic map) $\mu_{d,n} \colon \CP_d(n) \to \CP'_{d-1}$,
in much the same way as in the context of Newton's method.
Here we think of~$\CP'_{d-1}$ as the space of coefficient vectors
of the characteristic polynomials.

Our goal is now to find out if the image of~$\mu_{d,n}$ meets~$\CA_{d-1}$,
the set of polynomials all of whose roots are in the closed unit disk.

Since we expect that $\CP_d(n)$ is a finite-degree covering of~$\CP_d$,
it should in particular have dimension $\dim \CP_d = d-2$. This would imply
that the image of~$\mu_{d,n}$ has dimension at most \hbox{$d-2$}
(and we expect it to be exactly~$d-2$), so it is contained in a proper
algebraic subvariety of~$\CP'_{d-1}$. Each irreducible component
of~$\CP_d(n)$ will map to an irreducible component of this subvariety.
Such a subvariety of codimension at least~$1$ does not have to intersect a
given bounded subset like~$\CA_{d-1}$. This is a marked difference
compared to the situation with Newton's method, where the corresponding
multiplier map is surjective, and so examples of attracting $n$-cycles
can easily be found.

So our strategy will be to get as good control as we can on the
varieties~$\CP_d(n)$ (or suitable components of them), find the Zariski
closure~$X$ of their image under~$\mu_{d,n}$ and then check if $X$
meets~$\CA_{d-1}$. If it does not, then clearly no stable $n$-cycle can
exist on the component of~$\CP_d(n)$ that we are considering.
If it does, then we check that it also meets the open subset of~$\CA_{d-1}$
consisting of polynomials with all roots in the open unit disk;
then the intersection will contain a relative open
subset of~$X$ and so it will contain points in the image and such
that the corresponding polynomial~$p$ has distinct roots.


\section{Cycles for cubic polynomials} \label{Sec:Cubic}

We will now restrict consideration to cubic polynomials~$p$. Using affine
transformations, we can assume that $p(Z) = Z^3 + Z + t$ with some $t \in \C$.
This choice of parameterization excludes only (the affine equivalence classes
of) $Z^3 - 1$ (which corresponds to $t \to \infty$)
and the degenerate case~$Z^3$. The induced map to the true parameter
space~$\CP_3$ is a double cover identifying $t$ and~$-t$.
We will abuse notation slightly in the following by writing~$\CP_3(n)$
for what is really the pull-back of the true~$\CP_3(n)$ to the $t$-line
via the parameterization we use here.
As mentioned earlier, for such centered polynomials, the dynamics restricts
to the plane $H = \{z_1 + z_2 + z_3 = 0\}$.

Let $\sigma_k(\uz)$ denote the $k$-th elementary symmetric polynomial
in the entries of~$\uz$. We introduce the quantities
\[ w_2(\uz) = \sigma_2(\uz) - 1 \quad\text{and}\quad
   w_3(\uz) = \sigma_3(\uz) + t \,.
\]
(We shift by the elementary symmetric polynomials in the roots of~$p$
to move the image of the fixed points to~$(0,0)$.)
The map $\C^3 \supset H \to \C^2$ given by~$(w_2, w_3)$
has degree~$6$.

By the second property in Lemma~\ref{L:Weierstrass-1}, $W_p$ induces a
map~$\widetilde{W}_p$ on (a subset of) $\C^2$ such that
\[ \bigl(w_2(W_p(\uz)), w_3(W_p(\uz))\bigr)
     = \widetilde{W}_p\bigl(w_2(\uz), w_3(\uz)\bigr)
\]
for all $\uz \in (\C^3 \sm \Delta) \cap H$.

\begin{Lemma}
  $\widetilde{W}_p$ is given by
  \begin{align*}
    \widetilde{W}_p(w_2, w_3)
       = \frac{1}{\delta}
         \bigl(&w_2^2 + 2 w_2^3 - 3 w_3^2 - 9 t w_2 w_3 + w_2^4 + 6 w_2 w_3^2, \\
               &4 w_2 w_3 + 3t w_2^2 + 4 w_2^2 w_3 + 2t w_2^3 - 9t w_3^2 + w_2^3 w_3 + 8 w_3^3\bigr) \,,
  \end{align*}
  where
  \[ \delta = 4 (1 + w_2)^3 + 27 (t - w_3)^2 \,. \]
\end{Lemma}

\begin{proof}
  Routine calculation with a Computer Algebra System.
\end{proof}

Note that this explicit expression shows the quadratic convergence to~$(0,0)$
when $p$ has distinct roots, which is equivalent to $4 + 27t^2 \neq 0$.

Now suppose we have an $n$-cycle $(\uz^{(0)}, \uz^{(1)}, \ldots, \uz^{(n-1)})$
under~$W_p$. It will be attracting only if all eigenvalues of the multiplier
matrix $\TotalDiff{W_p^{\circ n}}{\uz^{(0)}}$ have absolute value at most~$1$.
Concretely, we consider the map $\mu_{3,n} \colon \CP_3(n) \to \CP'_2$
as discussed in Section~\ref{SS:perpts}. The characteristic polynomial will
have the form $Z^2 + c_1 Z + c_0$ with $c_0, c_1 \in \C$, and we know from
the discussion in Section~\ref{SS:perpts} that $c_0$ and~$c_1$ must satisfy
an algebraic relation, i.e., the points $(c_0, c_1)$ lie on some plane
algebraic curve as we run through all possible characteristic polynomials.

We can also consider the image of this $n$-cycle under~$(w_2, w_3)$, as
the map $\mu_{3,n}$ factors through the $(w_2, w_3)$-plane.
Assuming that $n$ is the minimal period of the cycle, the image cycle can
have minimal period~$n$, $n/2$ or~$n/3$. The second possibility occurs when
$n = 2k$ is even and $W_p^{\circ k}$
acts as a transposition on the vectors in the cycle. In this case, we say
that the cycle is of \emph{transposition type}.
The last possibility occurs when $n = 3k$ is divisible by~$3$ and $W_p^{\circ k}$
acts as a cyclic shift on the vectors in the cycle. In this case, we say that
the cycle has \emph{rotation type}. We can then equivalently look at the
characteristic polynomial of $\TotalDiff{\widetilde{W}_p^{\circ n}}{(w_2,w_3)(\uz)}$
(or with $k$ in place of~$n$ in the transposition or rotation type cases).

We will need a criterion that we can use to show that the two relevant
eigenvalues can never simultaneously be in the unit disk, in cases when
the relation between $c_0$ and~$c_1$ is somewhat involved. The following
lemma provides one such criterion.

\begin{Lemma} \label{L:test}
  Let $P(\lambda, \mu) \in \C[\lambda, \mu]$ be a polynomial. Fix a
  half-line $\ell$ emanating from the origin and some $N \in \Z_{>0}$.
  Let $B$ be the sum of the absolute values of the coefficients
  of the two partial derivatives of~$P$.
  If for all $j, k \in \{0,1,\ldots,N-1\}$, the distance from
  $P(e^{2\pi i j/N}, e^{2\pi i k/N})$ to~$\ell$ exceeds~$\pi B/N$,
  then $P(\lambda, \mu) = 0$ has no solutions in~$\C^2$ with
  \hbox{$|\lambda|, |\mu| \le 1$}.
\end{Lemma}

\begin{proof}
  We first show that the assumptions imply that the image of~$P$
  on the torus \hbox{$\Sphere^1 \times \Sphere^1$} is
  contained in the slit plane~$\C^2 \setminus \ell$.
  So consider $(u,v) \in [0,1]^2$ and pick $(j,k) \in \{0,\ldots,N\}^2$
  so that $|u - j/N|, |v - k/N| \le 1/(2N)$. Note that the sum of the
  absolute values of the partial
  derivatives of $(u, v) \mapsto F(u, v) := P(e^{2 \pi i u}, e^{2 \pi i v})$
  for $u,v \in \R$ is bounded by $2 \pi B$. This shows that
  \[ \bigl|P(e^{2 \pi i u}, e^{2 \pi i v}) - P(e^{2 \pi i j/N}, e^{2 \pi i k/N})\bigr|
       \le \frac{1}{2N} \|F_u\|_\infty + \frac{1}{2N} \|F_v\|_\infty
       \le \frac{1}{2N} \cdot 2\pi B
       = \pi B/N \,.
  \]
  Since the distance of $P(e^{2 \pi i j/N}, e^{2 \pi i k/N})$ from~$\ell$
  is by assumption larger than~$\pi B/N$, it follows that
  $P(e^{2 \pi i u}, e^{2 \pi i v}) \notin \ell$.

  We now assume that there is a solution with $|\lambda|, |\mu| \le 1$,
  so that the curve defined by~$P$ in~$\C^2$ meets the unit bi-disk.
  Since the curve is unbounded, by continuity there will be a solution with
  $|\lambda| = 1$ and $|\mu| \le 1$ or $|\mu| = 1$ and $|\lambda| \le 1$.
  By symmetry, we can assume the former. By the argument principle,
  the closed curve $\gamma \colon [0,1] \ni s \mapsto P(\lambda, e^{2\pi i s})$
  has to pass through the origin or wind around it at least once.
  However, since the assumptions imply that the image of~$\gamma$ is
  contained in the slit plane~$\C^2 \setminus \ell$, which
  does not contain the origin and is simply connected, we obtain
  a contradiction.
\end{proof}

The general procedure for obtaining the results given below is as follows.
\begin{enumerate}[1.]
  \item Set up equations for the variety~$\CP_3(n)$ or parts of
        it using~\eqref{Eq:Wgoodexplicit}.
  \item Set up the map $\mu_{3,n}$ as a map to the projective plane
        given by the coefficients of the characteristic polynomial
        of the multiplier matrix.
  \item \label{Step3}
        Use the Groebner Basis machinery of a Computer Algebra System
        like Magma~\cite{Magma} or Singular~\cite{Singular} to find
        the equation of the image curve.
  \item Either find a point on the image curve corresponding to a
        characteristic polynomial with both roots in the unit disk,
        or show using Lemma~\ref{L:test} that no such points exist.
\end{enumerate}

The available machinery can also be used to obtain additional information
on the components of the curves~$\CP_3(n)$, for example smoothness
or the (geometric) genus.

Since the map~$\mu_{3,n}$ is given by fairly involved rational functions
when $n$ is not very small, Step~\ref{Step3} above may not necessarily
be feasible as stated. In this case, we can instead sample some algebraic
points on the variety considered (e.g., by specializing the parameter~$t$
to a rational value and then determining the solutions of the resulting
zero-dimensional system) and consider their images under~$\mu_{3,n}$.
Given enough of these image points, we can fit a curve of lowest possible
degree through them (this is just linear algebra). We can then check
that this curve is correct by constructing a generic point on the original
variety and checking that its image lies indeed on the curve.

In the following, we always tacitly assume that the vectors occurring
in the cycles do not contain roots of~$p$. Those that do can easily
be described using Lemmas~\ref{L:Weierstrass-3} and~\ref{L:Weierstrass-4}.

The computations leading to the results given below have been done using
the Magma Computer Algebra System~\cite{Magma} and also in many cases
independently with Singular~\cite{Singular}. A Magma script containing
code that verifies most of the claims made is available at~\cite{Verification}.


\subsection{Points of order~$2$}

We begin by considering $2$-cycles. Note that a $2$-cycle of transposition type
fixes one component of the vector, which then must be a root of~$p$.
Since we have excluded cycles of this form (up to the obvious symmetries,
there are three of them, one for each root), no $2$-cycles of transposition
type have to be considered.

\begin{Proposition} \label{Prop:2c}
  The $2$-cycles form a smooth irreducible curve of geometric genus~$0$;
  it maps with degree~$12$
  to the $t$-line. So for each polynomial, there is (generically) one orbit
  of $2$-cycles under the natural action of $S_3 \times C_2$, where the first
  factor permutes the vector entries and the second factor performs a cyclic
  shift along the cycle. The image in $(t,w_2,w_3)$-space is the curve
  \[ w_2 = -3\,, \qquad 27 t^2 - 45 t w_3 + 20 w_3^2 - 20 = 0 \]
  of genus~$0$. The characteristic polynomial $X^2 + c_1 X + c_0$ of
  the multiplier matrix
  at a point on this curve satisfies the relation $c_0 + 2 c_1 + 6 = 0$.
  In particular, no $2$-cycle can be attracting.
\end{Proposition}

\begin{proof}
  This follows the method outlined above. Note that when both eigenvalues have
  absolute values at most~$1$, we have $|c_0| \le 1$ and $|c_1| \le 2$.
\end{proof}


\subsection{Points of order~$3$}

We begin by considering the $3$-cycles of rotation type. They can be defined
by~\eqref{Eq:Wgoodexplicit} together with $(z'_1, z'_2, z'_3) = (z_2, z_3, z_1)$
(for one choice of the cyclic permutation involved). Their images
under~$(w_2,w_3)$ are fixed points of~$\widetilde{W}_p$.

\begin{Proposition} \label{Prop:3r}
  The $3$-cycles of rotation type form two smooth irreducible curves
  (as $t$ varies) of geometric genus~$0$, 
  according to which of the possible two cyclic permutations results
  from the action of~$W_p$; the map to the $t$-line is of degree~$6$ in both cases.
  The images of both curves in $(t,w_2,w_3)$-space agree; the image curve
  is given by the equations
  \[ w_2 = -\frac{3}{2}\,, \qquad 216 t^2 - 360 t w_3 + 152 w_3^2 - 1 = 0 \,, \]
  describing a curve of genus~$0$. The characteristic polynomial
  of the multiplier matrix at a point on this curve (as a fixed point
  under~$\widetilde{W}_p$) has the form $X^2 + 3 X + a$
  for some $a \in \C$. In particular, such a $3$-cycle cannot be attracting.
\end{Proposition}

\begin{proof}
  This again follows the procedure outlined above. The characteristic
  polynomials lie on the curve $c_1 = 3$. So the sum of the eigenvalues is~$-3$,
  hence it is not possible that both eigenvalues are in the closed unit disk.
\end{proof}

Now we consider ``general'' $3$-cycles, i.e., $3$-cycles that are not
of rotation type.

\begin{Proposition} \label{Prop:3c}
  The $3$-cycles that are not of rotation type form two irreducible curves
  of geometric genus~$19$, which each map with degree~$72$ to the $t$-line
  and are interchanged by the action of any transposition in~$S_3$. Each curve
  therefore contains $8$~orbits of $3$-cycles under the action of $A_3 \times C_3$,
  and there are in total $8$~orbits under $S_3 \times C_3$, for each fixed~$t$.
  The coefficients
  $(c_0,c_1)$ of the characteristic polynomial of the multiplier matrix at
  a point in such a $3$-cycle give a point on a rational curve of degree~$12$
  that can be parameterized as $(c_0(u)/c_2(u), c_1(u)/c_2(u))$, where
  \begin{align*}
    c_0(u) &= -9 u^{12} - 162 u^{11} - 693 u^{10} + 1434 u^9 + 11958 u^8 - 32202 u^7 - 182301 u^6 \\
           &\qquad{} + 578742 u^5 + 2069910 u^4 - 919718 u^3 - 3065685 u^2 + 892254 u + 264295\,, \\
    c_1(u) &= u^{12} + 26 u^{11} + 230 u^{10} + 693 u^9 - 3867 u^8 - 5844 u^7 + 123074 u^6 \\
           &\qquad{} - 38381 u^5 - 1320149 u^4 + 420552 u^3 + 4310940 u^2 - 4206447 u + 1442574\,, \\
    c_2(u) &= -9 u^{10} - 63 u^9 + 301 u^8 + 1126 u^7 - 7693 u^6 - 3641 u^5 \\
           &\qquad{} + 52375 u^4 + 13526 u^3 - 104463 u^2 - 47919 u + 20987\,.
  \end{align*}
  In particular, no such $3$-cycle can be attracting.
\end{Proposition}

\begin{proof}
  The computations get quite a bit more involved, so we give more details
  here. We work in $7$-dimensional affine space over~$\Q$ with coordinates
  $(t, x_0, y_0, x_1, y_1, x_2, y_2)$, where the three vectors in the cycle
  are $\uz^{(j)} = (x_j, y_j, -x_j-y_j)$ for $j = 0,1,2$. We first set up the
  scheme giving the cycle $\uz^{(0)} \mapsto \uz^{(1)} \mapsto \uz^{(2)} \mapsto \uz^{(0)}$
  under~$W_p$. Then we remove the subschemes corresponding to cycles
  that have a fixed component or to $3$-cycles of
  rotation type. The resulting scheme is a curve mapping with degree~$144$
  to the $t$-line. Its projection to the $(x_2, y_2)$-plane is a curve
  of degree~$48$, whose defining polynomial factors into two irreducibles
  of degree~$24$ each that are interchanged by $x_2 \leftrightarrow y_2$.
  Let $Q$ denote one of the factors, considered as a bivariate polynomial.
  Since the projection is birational, this induces
  the splitting of the original curve into two components.
  We could compute the genus by working with the birationally
  equivalent plane curve given by $Q(x,y) = 0$. It has $222$ simple
  nodes (six of which are defined over~$\Q(\sqrt{6})$; the remaining
  $214$ are conjugate) and a pair of conjugate singularities defined
  over~$\Q(\sqrt{-3})$ that each contribute~$6$ to the difference
  between arithmetic and geometric genus. We obtain
  \[ g = \frac{23 \cdot 22}{2} - 222 - 2 \cdot 6 = 19 \]
  as claimed.

  This is a case where we had to use the sampling-and-interpolation trick
  to determine the image curve of~$\mu_{3,3}$ on one of the components.

  After showing that the image curve has geometric genus~$0$
  (there is one point of multiplicity~$4$ at~$(8, -9)$ that
  gives an adjustment of~$8$, and there are $47$ further simple nodes,
  so we obtain $g = 11 \cdot 10/2 - 8 - 47 = 0$) and finding some
  smooth rational points on it, we computed a parameterization
  modulo some large prime that maps $0, 1, \infty$ to three specified
  rational points and lifted it to~$\Q$. It is then easy to verify that
  we indeed obtain a parameterization of the curve over~$\Q$.
  We then used Magma's (fairly new and contributed by the third author)
  \texttt{ImproveParametrization} command to simplify the resulting
  parameterization.

  Finally, we use Lemma~\ref{L:test} to show that there is no attracting
  $3$-cycle (not of rotation type). We find the polynomial
  $P(\lambda, \mu) = 0$ that gives the relation between the eigenvalues
  $\lambda$ and~$\mu$
  (by substituting $(c_0, c_1) \leftarrow (\lambda\mu, -(\lambda+\mu))$
  in the equation relating the coefficients of the characteristic
  polynomial) and check that the criterion of Lemma~\ref{L:test}
  is satisfied when $\ell$ is the positive real axis and $N = 18$.
\end{proof}


\subsection{Points of order~$4$}

Judging by the heavy lifting that was necessary to deal with case of
general $3$-cycles, looking at general $n$-cycles with $n \ge 4$ seems too
daunting a task to attack with confidence along the lines described here.
We can, however, consider cycles with extra symmetries.
Here we look at $4$-cycles of transposition type.

\begin{Proposition} \label{Prop:4t}
  The $4$-cycles of transposition type form three irreducible smooth
  curves of geometric genus~$1$, each of degree~$24$ over the $t$-line,
  that are permuted
  by a cyclic shift of the coordinates. The characteristic polynomial
  $X^2 + c_1 X + c_0$ of the multiplier matrix at any associated point
  (considered as a point of order~$2$ under~$\tilde{W}_p$)
  satisfies the relation
  \begin{align*}
    34 c_0^4 c_1^2 + 169 c_0^3 c_1^3 - 675 c_0^2 c_1^4 - 2997 c_0 c_1^5
      - 2187 c_1^6 + 68 c_0^5 + 984 c_0^4 c_1 + 3359 c_0^3 c_1^2 & \\
    {} - 19182 c_0^2 c_1^3 - 88965 c_0 c_1^4 - 91584 c_1^5 + 4254 c_0^4
      + 29059 c_0^3 c_1 - 93688 c_0^2 c_1^2 \\
    {} - 634050 c_0 c_1^3 - 809379 c_1^4 + 76045 c_0^3 + 60846 c_0^2 c_1
      - 725626 c_0 c_1^2 - 1171592 c_1^3 \\
    {} + 487003 c_0^2 + 4167623 c_0 c_1 + 8653407 c_1^2 + 5442895 c_0
      + 15506760 c_1 - 35154225
     &= 0 \,,
  \end{align*}
  which describes a curve birationally equivalent to the elliptic
  curve over~$\Q$ with Cremona label~$15a4$.
  In particular, there do exist values of the parameter~$t$ such that
  there are attracting $4$-cycles of transposition type. Such parameters
  can be found near
  \[ t \approx 177.68741192204597 \,. \]
\end{Proposition}

\begin{proof}
  We set up the variety describing $4$-cycles of transposition type
  as a subscheme of $5$-dimensional affine space with coordinates
  $t, x_0, y_0, x_1, y_1$, where $t$ is the parameter and the iteration
  satisfies
  \[ (x_0, y_0, -x_0-y_0) \longmapsto (x_1, y_1, -x_1-y_1) \longmapsto (y_0, x_0, -x_0-y_0) \,, \]
  and we remove the component consisting of cycles in which the last
  coordinate is fixed. This results
  in a smooth irreducible curve of degree~$24$ over the $t$-line
  that has genus~$1$. 
  We find the image curve in the $(c_0, c_1)$-plane. We compute
  that the geometric genus of the image curve is~$1$ and find
  a smooth rational point on it. This allows us to identify the
  elliptic curve it is birational to. From the explicit
  equation, we find that there is a characteristic polynomial that has
  a double root near~$-0.68916660883309$. This leads to the given
  value of~$t$ (and its negative).
\end{proof}

\begin{Remark}
  The region in the $t$-plane consisting of parameter values for which
  an attracting $4$-cycle of transposition type exists is a union of
  two components, mapped to each other by $t \mapsto -t$. Each of them
  is symmetric with respect to the real axis; the component containing
  values with positive real part is shown in Figure~\ref{Fig:4cycle} in blue.

  One can verify numerically that as $t$ increases along the real
  axis beyond the boundary of this region, a symmetry-breaking bifurcation
  occurs, and we find an adjacent region where attracting general
  $4$-cycles (i.e., not of transposition type) exist. This region
  is shown in green in Figure~\ref{Fig:4cycle}.

  In Figure~\ref{Fig:parameterspace} we show how these regions are
  located relative to the parameter space of cubic Newton maps,
  in terms of the parameterization that is more commonly used in this context.
  It is apparent that these regions in parameter space are quite small.
  In addition, the left part of Figure~\ref{Fig:4cycle} shows that
  the basin of attraction of the attracting $4$-cycles is also quite
  small as a subset of the dynamical plane. It is therefore not very
  surprising that examples of polynomials for which the Weierstrass method
  exhibits attractive cycles had not been found previously by numerical
  methods.

  It is well known that the parameters~$\lambda$ for which the Newton
  map has attracting cycles of period~$2$ or greater are organized in
  the form of little Mandelbrot sets, finitely many for each period,
  and that every parameter in the bifurcation locus (common boundary
  point of any two colors) contains, in every neighborhood, infinitely
  many such little Mandelbrot sets. In Figure~\ref{Fig:mandelbroetchen}
  we compare with one of these regions in parameter space where attractive
  $4$-cycles exist for Newton's method.
  This period~$4$ component ranges from imaginary parts $0.6240$
  to~$0.6267$ along the line $\Re(z) = 0.5$, hence is of diameter
  about~$0.0027$; for comparison:
  the period~$4$ component for Weierstrass has imaginary parts between
  $0.8467$ and~$0.8481$, hence diameter about~$0.0014$, which is roughly
  comparable (even though there is no uniform Euclidean scale across
  parameter space).
\end{Remark}

\begin{figure}[htb]
  \includegraphics[width=.4\textwidth]{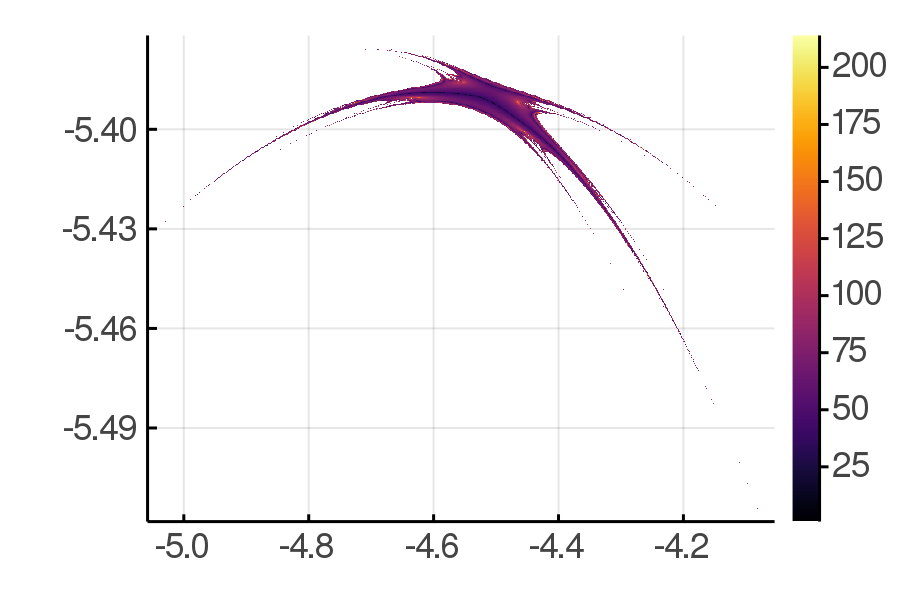}
  \qquad
  \includegraphics[width=.45\textwidth]{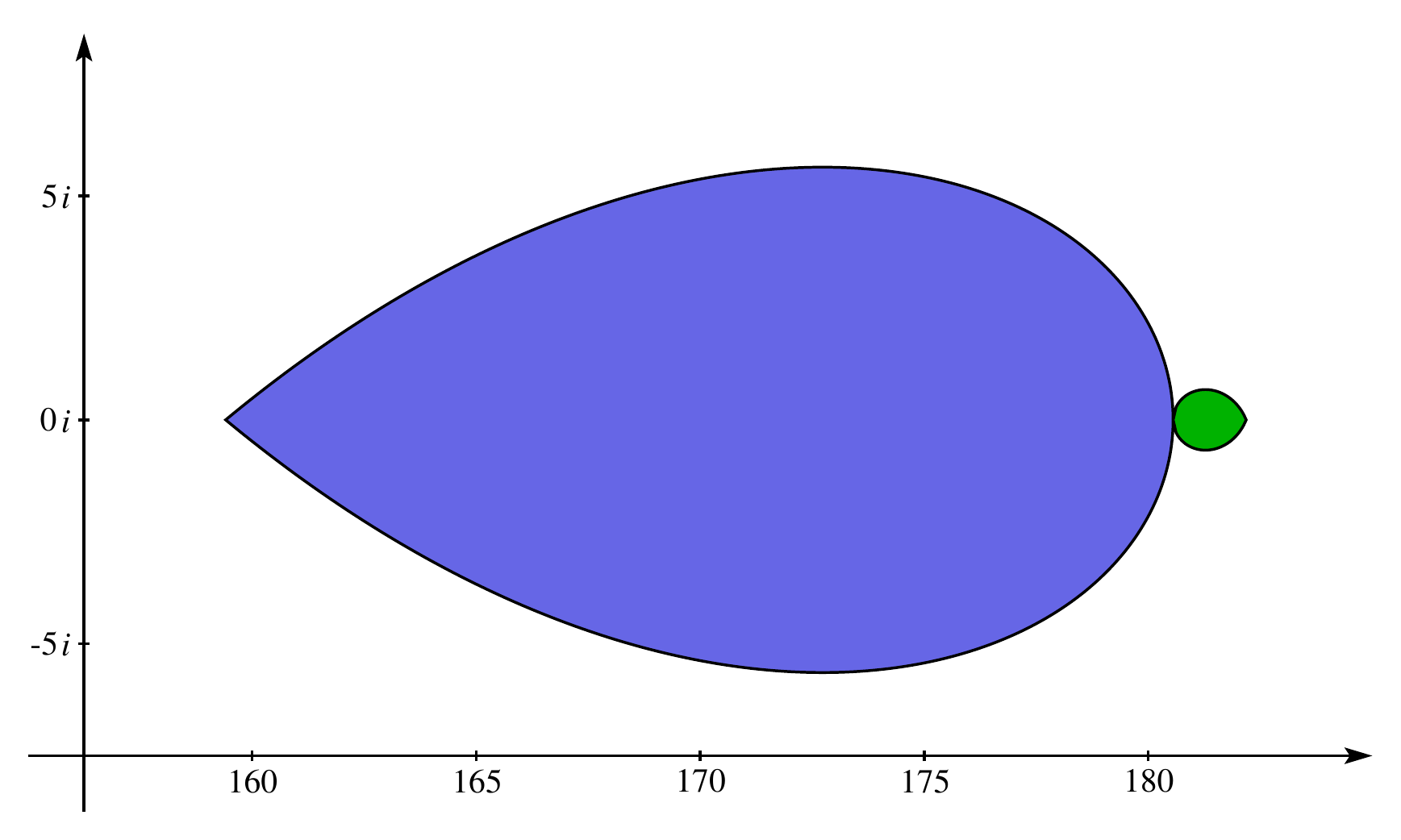}
  \caption{\textbf{Left:} An approximation to the real section of the immediate basin
           of attraction around a periodic point of order~$4$ for
           $W_{Z^3+Z+t}$ with $t = 177.68741192204597$.
           The coordinates shown are $(z_1, z_2)$;
           the corresponding point is $(z_1, z_2, -z_1-z_2)$.
           The coloring encodes the number of iteration steps necessary
           to get within distance~$10^{-4}$ of the periodic point.
           \textbf{Right:} Parameter values~$t \in \C$ for which there exists a stable
           $4$-cycle of transposition type (left region, blue)
           or a stable $4$-cycle without extra symmetry (right region, green).
           The components touch at a point where the multiplier matrix
           under~$\widetilde{W}_p^{\circ 2}$ has eigenvalue~$-1$.}
  \label{Fig:4cycle}
\end{figure}

\begin{figure}[htbp]
  \includegraphics[height=.4\textwidth]{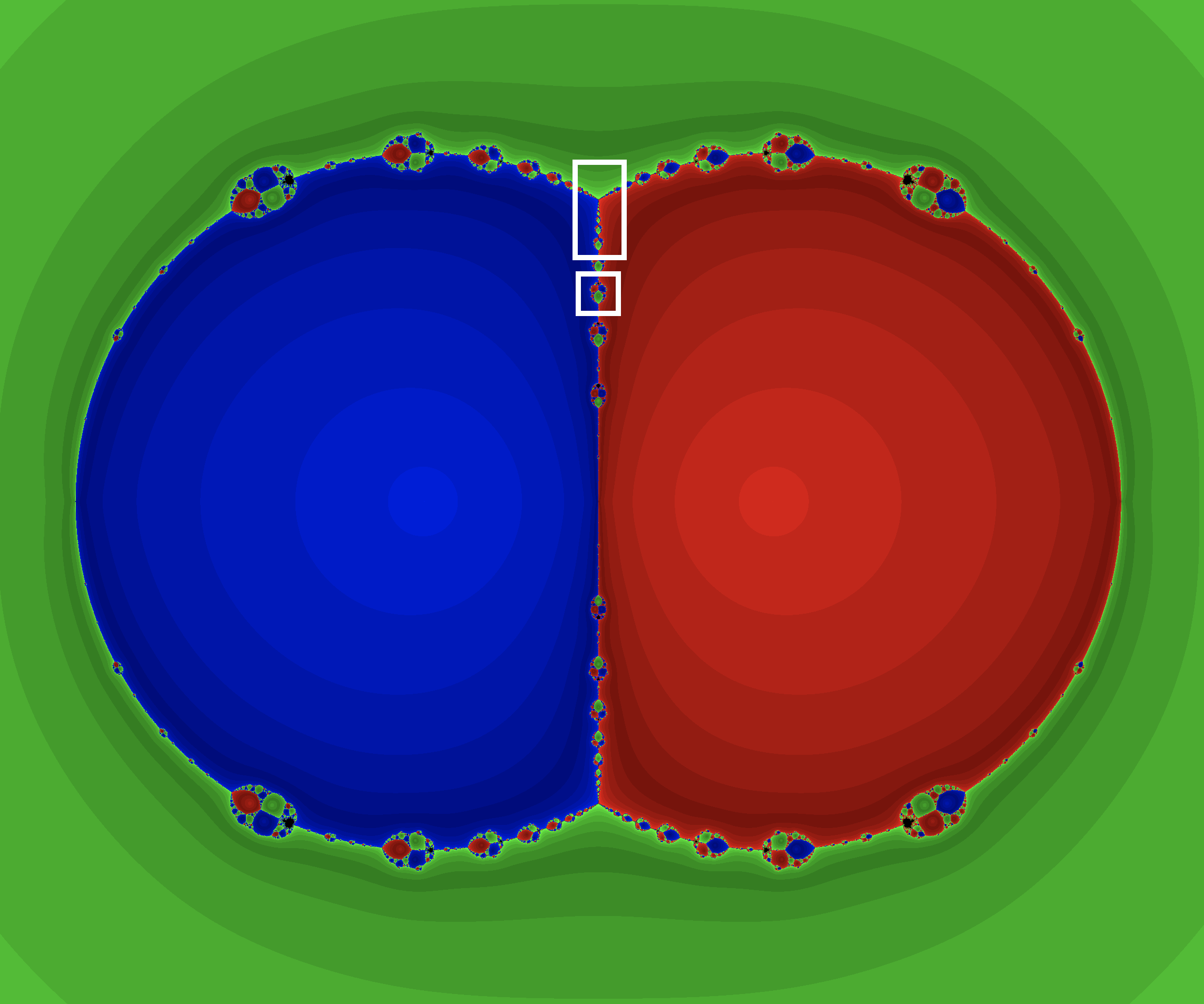}
  \qquad
  \includegraphics[height=.4\textwidth]{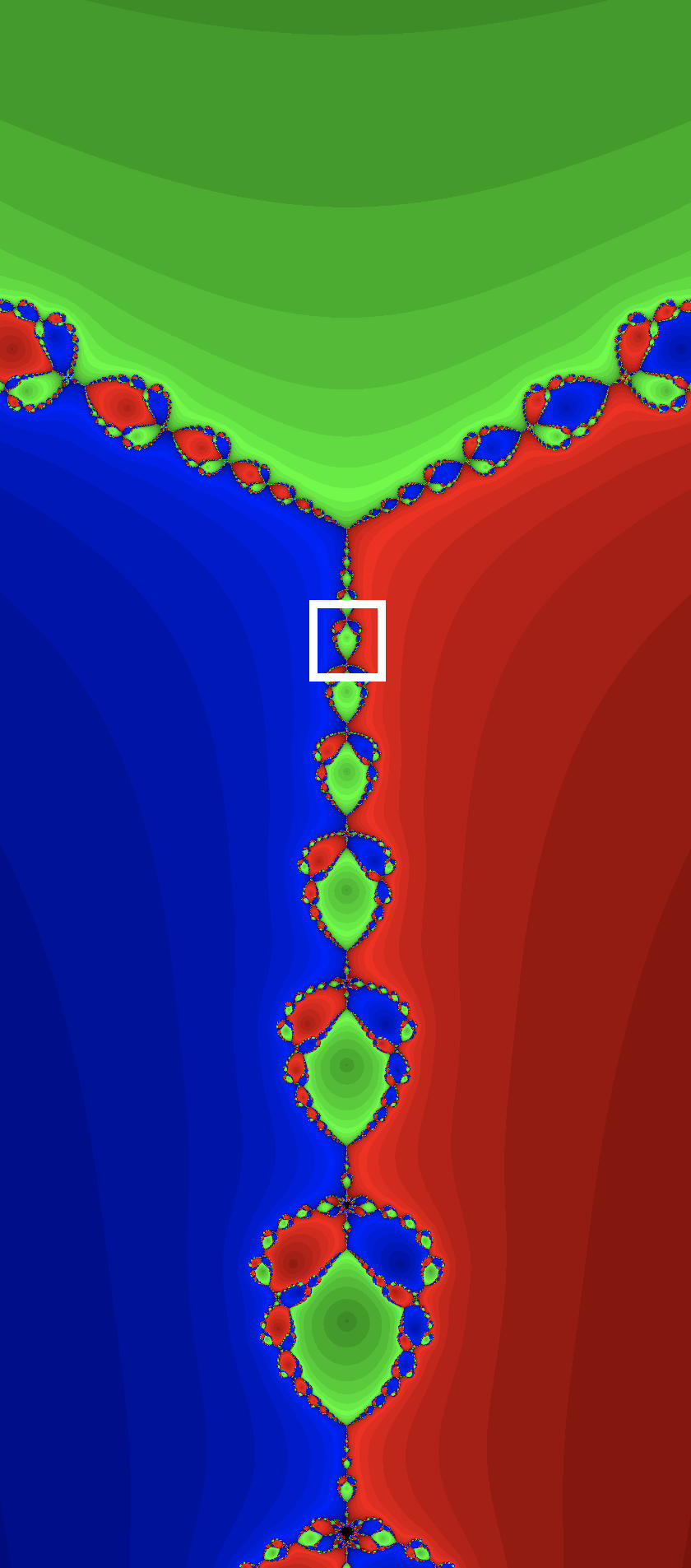}
  \qquad
  \includegraphics[height=.4\textwidth]{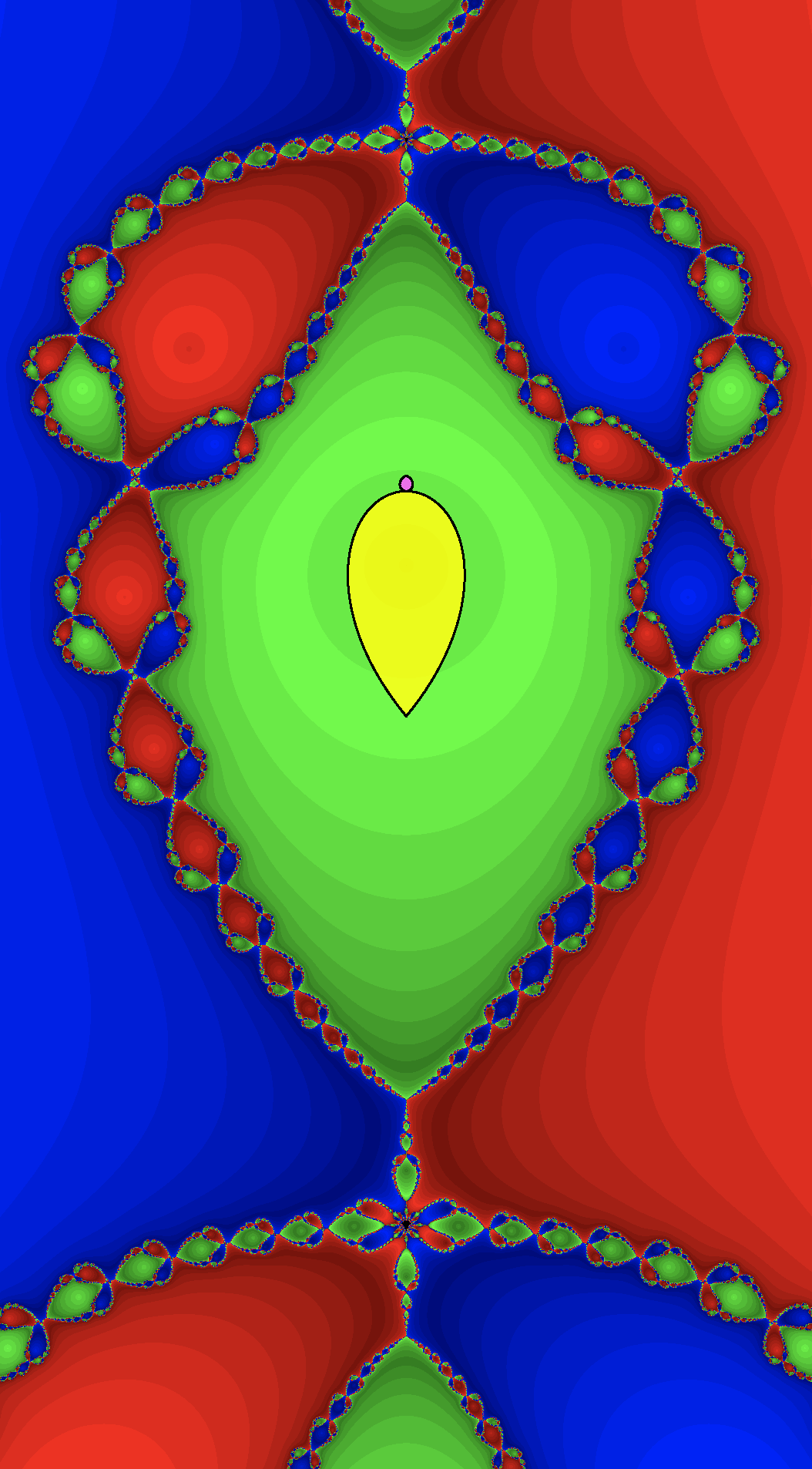}
  \caption{The parameter space of cubic polynomials up to affine precomposition,
           parameterized as $p(Z) = Z (Z-1) (Z-\lambda)$ with $\lambda \in \C$;
           shown is the complex $\lambda$-plane. This is the standard
           parameterization used in previous research on the Newton dynamics.
           The picture illustrates the Newton dynamics: the three colors indicate
           to which of the three roots $0$, $1$, and~$\lambda$ the free critical
           point converges. \newline
           The picture on the left shows a global view of parameter space,
           with two subsequent magnifications shown in the middle and on the right.
           The regions shown on the right in Figure~\ref{Fig:4cycle}, which
           indicate parameter values for which an attractive $4$-cycle exists
           for the Weierstrass iteration, are superimposed on the
           last magnification (shown in yellow and magenta and converted
           to the different parameterization used here).}
  \label{Fig:parameterspace}
  \vspace{15pt}
  \includegraphics[height=.4\textwidth]{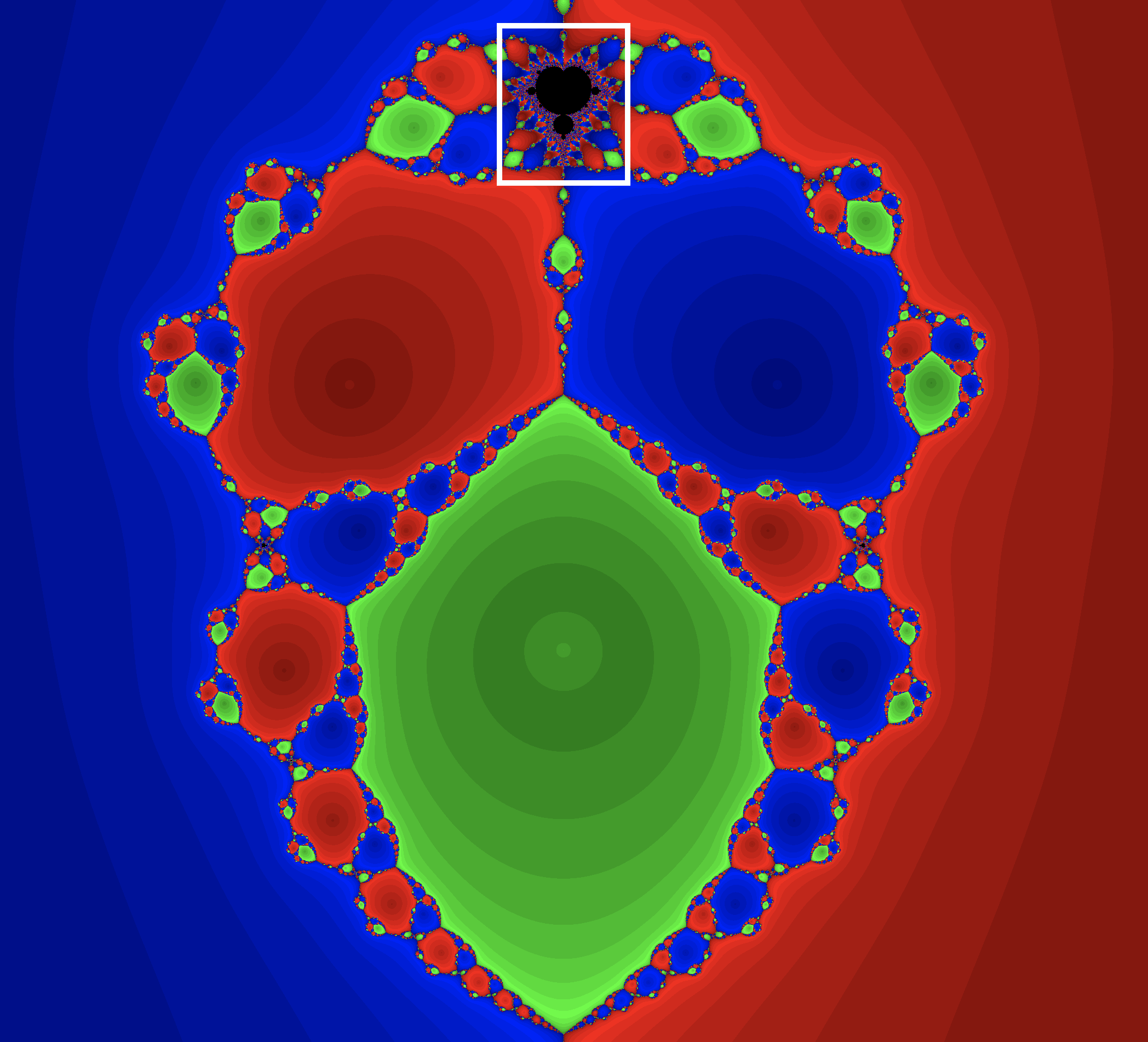}
  \qquad
  \includegraphics[height=.4\textwidth]{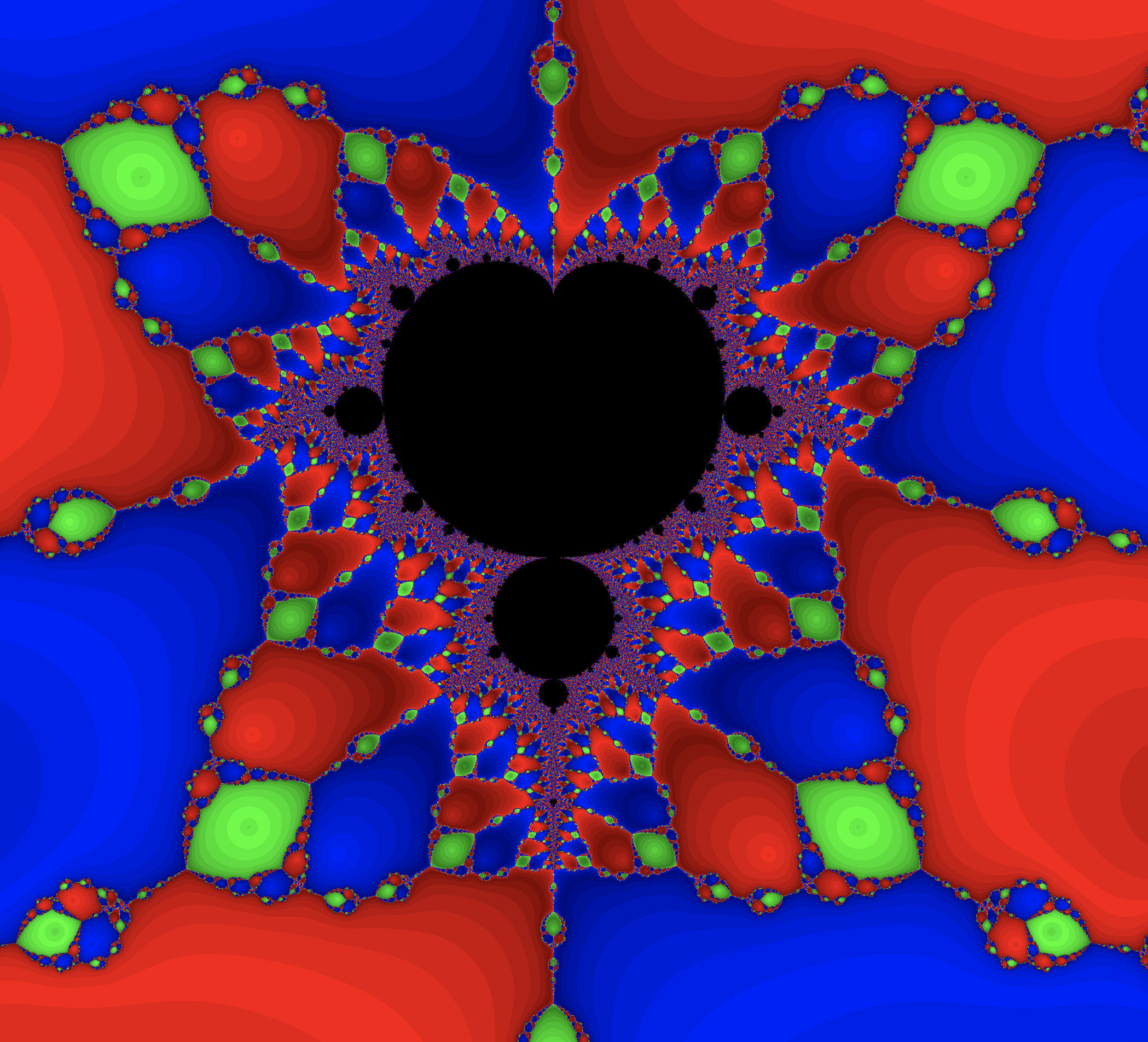}
  \caption{Sequence of refinements, starting with the first picture in
           Figure~\ref{Fig:parameterspace}, towards the largest
           of the ``little Mandelbrot sets'' around attracting cycles of period~$4$
           for the Newton iteration (there are apparently larger little Mandelbrot
           sets on the outer circles; these represent conformally conjugate
           dynamics and seem larger only because of different parametrizations).}
  \label{Fig:mandelbroetchen}
\end{figure}


\subsection{Points of order~$6$}

Finally, we consider $6$-cycles of rotation type.

\begin{Proposition} \label{Prop:6r}
  The $6$-cycles of rotation type form two irreducible smooth
  curves of geometric genus~$5$, each of degree~$24$ over the $t$-line,
  that are permuted by a transposition of the coordinates.
  The characteristic polynomial
  $X^2 + c_1 X + c_0$ of the multiplier matrix at any associated point
  (considered as a point of order~$2$ under~$\tilde{W}_p$)
  satisfies a relation that specifies a curve of geometric genus~$0$
  and degree~$5$. This curve
  can be parameterized as $(c_0(u)/c_2(u), c_1(u)/c_2(u))$, where
  \begin{align*}
    c_0(u) &= -36 u^5 - 12 u^4 + 60 u^3 + 236 u^2 + 260 u - 4\,, \\
    c_1(u) &= -4 u^5 - 51 u^4 - 90 u^3 + 59 u^2 + 42 u + 5\,, \\
    c_2(u) &= -9 u^4 - 18 u^3 + u^2 + 10 u - 1\,.
  \end{align*}
  In particular, no such $6$-cycle can be attracting.
\end{Proposition}

\begin{proof}
  We set up the variety describing $6$-cycles of rotation type
  as a subscheme of $5$-dimensional affine space with coordinates
  $t, x_0, y_0, x_1, y_1$, where $t$ is the parameter and the iteration
  satisfies
  \[ (x_0, y_0, -x_0-y_0) \longmapsto (x_1, y_1, -x_1-y_1) \longmapsto (y_0, -x_0-y_0, x_0) \,, \]
  and we remove components coming from
  $3$-cycles of rotation type. This results
  in a smooth irreducible curve of degree~$24$ over the $t$-line
  that has genus~$5$.
  We find the image curve in the $(c_0, c_1)$-plane. Since the
  degree and the coefficient size are moderate, we can directly check that
  the curve has geometric genus~$0$ and then find a parameterization.
  We then use the explicit equation and Lemma~\ref{L:test} with $\ell$
  the negative real axis and $N = 12$ to verify that no characteristic
  polynomial lying on the curve can have both roots in the unit disk.
\end{proof}


\subsection{Proof of Theorem~\ref{MT1}}

The results obtained in this section provide a proof of part~(\ref{Item:MT1-CubicOkay})
of Theorem~\ref{MT1}. Proposition~\ref{Prop:4t} gives a proof
of part~(\ref{Item:MT1-NotGenConv}) for the case $d = 3$. To obtain the conclusion for
all $d \ge 3$, we invoke Lemma~\ref{L:lift}.


\end{document}